\documentclass[11pt]{article}

\usepackage{graphicx}
\usepackage{amsmath}
\usepackage{amsfonts}
\usepackage{amssymb}
\usepackage{amsthm}
\usepackage{a4wide}

\usepackage{siunitx}

\newcommand{\Ev}{\mathbf{E}}
\newcommand{\Dv}{\mathbf{D}}
\newcommand{\Pv}{\mathbf{P}}

\newcommand{\uv}{\mathbf{u}}
\newcommand{\nv}{\mathbf{n}}
\newcommand{\tv}{\mathbf{t}}
\newcommand{\fv}{\mathbf{f}}
\newcommand{\St}{\mathbf{S}}
\newcommand{\Tt}{\mathbf{T}}
\newcommand{\It}{\mathbf{I}}

\newcommand{\betat}{\boldsymbol{\beta}}

\newcommand{\weakconv}{{\rightharpoonup}}
\newcommand{\weakstarconv}{\overset{\ast}{\rightharpoonup}}
\newcommand{\cten}{\mathbf{c}}
\newcommand{\dten}{\mathbf{d}}

\newcommand{\hten}{\mathbf{h}}

\newcommand{\dissipation}{\mathcal{D}}
\newcommand{\subdiff}{\partial}

\newcommand{\spaceU}{\mathbb{U}}
\newcommand{\spaceW}{\mathbb{W}}
\newcommand{\spaceV}{\mathbb{H}}
\newcommand{\spaceD}{\mathbb{D}}
\newcommand{\spaceS}{\mathbb{S}}
\newcommand{\spaceP}{\mathbb{P}}

\newcommand{\spaceX}{\mathbb{X}}
\newcommand{\Rbarplus}{\overline{\mathbb{R}}^+}
\newcommand{\ww}{\mathbf{w}}
\newcommand{\uu}{\mathbf{w}}
\newcommand{\zzz}{\mathbf{z}}

\newtheorem{theorem}{Theorem}
\newtheorem{definition}{Definition}
\newtheorem{lemma}{Lemma}

\usepackage{color}

\newcommand{\opdiv}{\operatorname{div}}
\newcommand{\opcurl}{\operatorname{curl}}

\newcommand{\oldnew}[2]{{#2}}

\begin{document}
	
	\title{The polarization process of ferroelectric materials analyzed in the framework of variational inequalities}
	
	\author{Astrid S. Pechstein*, Martin Meindlhumer and Alexander Humer}

	\date{July 2019}
	\maketitle
	
	\abstract{We are concerned with the mathematical modeling of the polarization process in ferroelectric media. We assume that this dissipative process is governed by two constitutive functions, which are the free energy function and the dissipation function. The dissipation function, which is closely connected to the dissipated energy, is usually non-differentiable.  Thus, a minimization condition  for  the  overall  energy  includes  the subdifferential of  the  dissipation  function.   This condition can also be formulated by way of a variational inequality in the unknown fields strain,  dielectric  displacement,  remanent polarization  and remanent strain. 
		We analyze the mathematical well-posedness of this problem.
		We provide an existence and uniqueness result for the time-discrete update equation. Under stronger assumptions, we can prove existence of a solution to the time-dependent variational inequality.
		To solve the discretized variational inequality, we use mixed finite elements, where mechanical displacement and dielectric displacement are unknowns,  as well as polarization (and,  if included in the model,  remanent strain). It is then possible to satisfy Gauss' law of zero free charges exactly.  We propose to regularize the dissipation function and solve for all unknowns at once in a single Newton iteration. We present numerical examples gained in the open source software package Netgen/NGSolve.}

\section{Introduction} \label{sec:intro}

A thermodynamical framework for the description of ferroelectric materials based on the Helmholtz free energy was originally provided in the series of papers \cite{BassiounyGhalebMaugin:1988a,BassiounyGhalebMaugin:1988b,BassiounyMaugin:1989a,BassiounyMaugin:1989b}
by Bassiouny, Ghaleb and Maugin. Their theory allows to describe multiaxial electromechanical loading procedures.
The introduced notions are similar to elasto-plasticity, including internal variables, yield (or switching) criteria and hardening moduli.
Explicit choices of energy and switching criteria were provided by Cocks and McMeeking \cite{CocksMcMeeking:1999} for the one-dimensional case.
The multi-dimensional case followed in the works of McMeeking and Landis \cite{McMeekingLandis:2002} and Landis \cite{Landis:2002}. In the former reference, the remanent polarization vector is the only internal unknown, and the polarization stress is linked directly to the remanent polarization. Contrarily, in the latter work remanent polarization and strain are independent of each other, but are determined by a common switching condition. With these theories, not only hysteresis loops can be tracked, but also butterfly hystereses are predicted correctly. The models were validated against measurements provided by Huber and Fleck \cite{HuberFleck:2001} for non-proportional loading procedures.

Another approach based on the thermodynamic framework due to the group around Maugin is that by Kamlah and Tsakmakis \cite{KamlahTsakmakis:1999}, see also \cite{Kamlah:2001}. They use a set of different switching and saturation conditions to determine the evolution of remanent polarization and polarization strain.

Miehe, Rosato and Kiefer \cite{MieheRosatoKiefer:2011} introduced an incremental variational principle for the even more general case of coupled electro-magneto-mechanics. They distinguish between \emph{energy}-based and \emph{enthalpy}-based models. For the latter, the independent unknowns are strain and electric field, whereas for the former, strain and dielectric displacement are independent. While most finite element formulations use enthalpy-based models discretizing the electric potential, we provide theory for an energy-based setting with an independent dielectric displacement.

An uni-dimensional energy-based model is presented by Sands and Guz \cite{SandsGuz:2013}. Semenov et al.\ \cite{SemenovKesslerLiskowskyBalke:2006} present a three-dimensional energy-based formulation using a vector potential for the dielectic displacement. Then, the dielectric displacement vector satisfies Gauss' law of zero divergence automatically. Contrarily, we propose to use $H(\opdiv)$ conforming finite elements as can be found in the context of mixed methods \cite{BoffiBrezziFortin:2013}. Gauss' law is then introduced as a constraint enforced exactly by a Lagrangian multiplier. This way accuracy of the electric unknowns is improved, as no derivatives have to be taken.
In some sense, the approach by Klinkel \cite{Klinkel:2006} may be seen as diametrically opposed, where an irreversible electric field is introduced instead of the remanent polarization vector.

\oldnew{}{All above mentioned models are macroscopic phenomenological ones. A different approach is  modeling on the micro- or meso-scale, we cite an early model by Hwang et al.\ \cite{HWANG19952073}. Computations on unit cells in nano-meter range, where polarization domains have to be resolved by the finite element mesh, represent the homogenized material behavior. This leads to more complex material models involving different internal variables, usually providing a different -- probably higher -- level of accuracy at higher cost as compared to the purely phenomenological models, as was observed e.g.\ by \cite{LinMuliana:2014,jayendiran2016finite}. }

In the current paper, we aim at proving existence and uniqueness of a solution to the problem of finding an update solution in the polarization process of ferroelectric media.
To this end, we describe the polarization of ferroelectric media as a dissipative process in the mathematical framework of variational inequalities. We start from an energy-based model of the problem, from which we derive time-dependent variational equations and inequalities. 
The independent unknowns are then strain (or displacement), dielectric displacement, remanent polarization, and, if included in the model, remanent polarization strain.
We provide an abstract mathematical framework for this problem. \oldnew{}{For a comprehensive overview on the mathematical modeling of dissipative systems we refer to the monograph by Mielke and Roub\'{\i}\v{c}ek \cite{MielkeRoubicek:2015}.} Similar abstract problems have been analyzed in context of elasto-plasticity (see e.g. the monograph by Han and Reddy \cite{HanReddy:1999}) or contact mechanics (we refer to Sofonea and Matei \cite{SofoneaMatei:2011}). We proceed in a similar manner as in the first reference \cite{HanReddy:1999} and see that, under the standard assumption of convexity of the free energy function, existence and uniqueness of a solution to the time-discrete update problem can be shown. Two standard material models are demonstrated to fit into this framework.
These results fit well with the findings in \cite{StarkNeumeisterBalke:2016a,BotteroIdiart:2018}, where stability issues are discussed.
Under further assumptions on the energy, also existence of a time-dependent solution can be guaranteed. However, we see that these stronger assumptions are only satisfied for very simple models not including saturation. Whether it is possible to weaken these requirements will be subject of further research.

This paper is organized as follows: in Section~\ref{sec:modeling} the underlying energy-based consitutive models are described, and remanent quantities, dissipation function and dissipative driving forces are introduced on a physical level. These quantities are embedded into an abstract mathematical framework of variational inequalities in Section~\ref{sec:framework}, where also all assumptions are stated. The time-discrete update equation is introduced in Section~\ref{sec:timediscr}. Existence and uniqueness of a solution update are shown. In Section~\ref{sec:timedep}, this result is used to gain existence of a time-dependent solution to the original variational inequality under stronger assumptions. Two standard material models are analyzed with respect to the abstract theory in Section~\ref{sec:application}. Computational aspects of a finite element implementation are discussed in Section~\ref{sec:fe}, and numerical results are provided in Section~\ref{sec:numerics}.

\section{Energy-based constitutive modeling} \label{sec:modeling}

In the following, we present a variational inequality that describes the problem of polarization of ferroelectric media. We postpone mathematical exactness concerning solution spaces, (weak) differentiability and other issues to Section~\ref{sec:framework}, in order not to complicate matter too much. Let $\Omega \subset \mathbb{R}^d, d=2,3$ denote the body of interest. Concerning the mechanical quantities, we use $\uv$ for the displacement field and $\Tt$ for the total stress. We assume small deformations, thus we identify deformed and undeformed configuration, and use the linear strain $\St = \frac{1}{2} (\nabla \uv + \nabla \uv^T)$. Additionally, we assume a quasistatic regime. The electric part of the problem is then characterized by the electric potential $\phi$, while $\Ev = - \nabla \phi$ is the electric field. The dielectric displacement vector shall be denoted as $\Dv$. Last, in the polarization problem we are interested in finding the remanent polarization $\Pv^{i}$ and the remanent strain $\St^{i}$.

We assume that the material is characterized by a Helmholtz free energy function 
\begin{align}
\Psi &= \Psi^r(\Dv, \Pv^{i}, \St, \St^{i}) + \Psi^{i}(\Pv^{i}, \St^{i}),
\end{align}
which consists of a \emph{reversible} or \emph{stored} part $\Psi^r$, and a part $\Psi^{i}$ that is associated only with the internal variables of remanent polarization and strain. This part may tend to infinity as the polarization approaches saturation. For a similar characterization, we refer to \cite{Landis:2002}. 

The independent reversible unknowns are strain and dielectric displacement, where the latter additionally has to satisfy Gauss' law,
\begin{align}
\opdiv \Dv &= 0.
\end{align}
Electric field and stress are dependent quantities, and defined as derivatives of the free energy with respect to dielectric displacement and strain,
\begin{align}
\Ev &= \frac{\partial \Psi}{\partial \Dv}, & 
\Tt &= \frac{\partial \Psi}{\partial \St}. \label{eq:et}
\end{align}

As strain and dielectric displacement are not constrained, or constrained to the linear subspace of divergence-free functions in the latter case, there holds the following variational equation,
\begin{align}
\Ev \cdot \delta \Dv + \Tt \cdot \delta \St &= \delta W_{ext} && \text{for all } \delta \Dv, \opdiv \delta \Dv = 0 \text{ and } \delta \St = \St(\delta \uv). \label{eq:virtualworks}
\end{align}
In \eqref{eq:virtualworks}, all virtual strains and divergence-free virtual dielectric displacements are considered, where $\delta \Dv$ and $\delta \uv$ have to satisfy the respective boundary conditions. On the right hand side, all virtual work by external forces is summarized as $\delta W_{ext}$. This formula is well-known as \emph{principle of virtual works}.

Polarization of ferroelectric materials is a dissipative process. The \emph{dissipative driving forces} dual to the irreversible quantities are
\begin{align}
\hat \Ev &= -\frac{\partial \Psi}{\partial \Pv^{i}}, &
\hat \Tt &= -\frac{\partial \Psi}{\partial \St^{i}}.  \label{eq:hatehatt}
\end{align}
The dissipation is then given by the inner product of these driving forces and the \emph{dissipative fluxes} $\dot \Pv^{i}$ and $\dot \St^{i}$,
\begin{align}
\dissipation = \hat \Ev \cdot \dot \Pv^{i} + \hat \Tt : \dot \St^{i}.
\end{align}
The dissipation function $\Phi$ relates the driving forces $\hat \Ev$, $\hat \Tt$ to the driving rates $\dot \Pv^{i}$, $\dot \St^{i}$. Typically, the dissipation function is not smooth, but weakly lower semicontinuous and allows for a \emph{subdifferential}\footnote{The subdifferential of some function $j$ with respect to $u$ is denoted by $\subdiff_u j$ and represents a set, namely
	\begin{align}
	z& \in \subdiff_u j(u) & \Longleftrightarrow && j(v) - j(u) &\geq z\cdot (v-u) & \text{for all }v.
	\end{align}}.
The driving forces are contained in the respective subdifferentials, i.e.
\begin{align}
\hat \Ev &\in \partial_{\dot \Pv^{i}} \Phi(\dot \Pv^{i}, \dot \St^{i}) &  \text{and} &&
\hat \Tt &\in \partial_{\dot \St^{i}} \Phi(\dot \Pv^{i}, \dot \St^{i}). \label{eq:subdiff}
\end{align}
By definition of the subdifferential, eq. \eqref{eq:subdiff} is equivalent to the variational inequality
\begin{align}
-\hat \Ev \cdot (\tilde \Pv^{i} - \dot \Pv^{i}) - \hat \Tt \cdot (\tilde \St^{i} - \dot \St^{i}) + \Phi(\tilde \Pv^{i}, \tilde \St^{i}) - \Phi(\dot \Pv^{i}, \dot \St^{i}) \geq 0 && \forall \tilde \Pv^{i}, \tilde \St^{i}. \label{eq:ineq_pi}
\end{align}
To arrive at one single variational inequality, we add up \eqref{eq:virtualworks} and \eqref{eq:ineq_pi}. In \eqref{eq:virtualworks}, we use the admissible virtual dielectric displacement $\delta \Dv = \tilde \Dv - \dot \Dv$, where $\opdiv \tilde \Dv = \opdiv \dot \Dv = 0$. The virtual strain is defined from the admissible displacement update $\tilde \uv$ and $\dot \uv$ by $\delta \St = \St(\tilde \uv) - \St(\dot \uv)$. With these choices, we deduce for all $\tilde \uv, \tilde \Dv, \tilde \Pv^{i}$ and $\tilde \St^{i}$,
\begin{align}
\begin{split}
\Ev \cdot (\tilde \Dv - \dot \Dv) + \Tt : (\St(\tilde \uv) - \dot \St) - \hat \Ev \cdot (\tilde \Pv^{i}- \dot \Pv^{i}) - \hat \Tt : (\tilde \St^{i} - \dot \St^{i})&\\ + \Phi(\tilde \Pv^{i}, \tilde \St^{i}) - \Phi(\dot \Pv^{i}, \dot \St^{i}) &\geq \delta W^{ext}. 
\end{split}\label{eq:varin_phys}
\end{align}

\section{A mathematical framework} \label{sec:framework}
We define the variational inequality that describes the polarization process in a mathematical framework. Therefore, we introduce compact notation, which is compatible with the literature on variational inequalities, especially with the monograph \cite{HanReddy:1999}. We use
\begin{align}
\ww &= [\St(\uv), \Dv, \St^{i}, \Pv^{i}]^T && \text{for the reactions}\\
\zzz &= [\St(\tilde \uv), \tilde \Dv, \tilde \St^{i}, \tilde \Pv^{i}]^T && \text{for admissible rates}. 
\end{align}
A priori, we assume the different quantities to live in the following Hilbert spaces,
\begin{align}
\uv \in \spaceU &:= \{\uv \in [H^1(\Omega)]^d: \uv = \mathbf{0} \text{ on } \Gamma_{fix}\},\\
\Dv \in \spaceD_0 &:= \{\Dv \in\spaceD: \opdiv \Dv = 0\} \nonumber\\
& \text{with} \ \spaceD := \{\Dv \in H(\opdiv,\Omega): \Dv\!\cdot\!\nv = 0 \text{ on } \Gamma_{ins}\},\\
\St^{i} \in \spaceS &:=[L^2(\Omega)]^{d\times d}_{sym},\\
\Pv^{i} \in \spaceP &:=[L^2(\Omega)]^d.
\end{align}
Compound spaces for the compact notation are then,
\begin{align}
\spaceV &:= \spaceU \times \spaceD \times \spaceS \times \spaceP, &
\spaceV_0 &:=  \spaceU \times \spaceD_0 \times \spaceS \times \spaceP.
\end{align}
Depending on the definition of the irreversible energy $\Psi^{i}$, the free energy $\Psi$ may tend to infinity as the material approaches polarization saturation. We introduce the \emph{effective domain} of $\Psi$ by (cf. \cite[p~73]{HanReddy:1999})
\begin{align}
\spaceX &:= \{ \zzz \in \spaceV: \Psi(\zzz) < \infty\}, & \spaceX_0 := \spaceX \cap \spaceV_0.
\end{align}

We define the nonlinear operator $A: \spaceX \to \spaceV^*$ by its action
\begin{align}
\langle A(\ww), \zzz \rangle &:= \Big\langle \frac{\partial \Psi}{\partial \ww}(\ww) , \zzz \Big\rangle. 
\end{align}
Note that, in physical meaning according to \eqref{eq:et} and \eqref{eq:hatehatt}, the operator $A$ maps reactions to forces, i.e.
\begin{align}
A(\ww) = A(\St(\uv), \Dv, \St^{i}, \Pv^{i}) = [\St(\uv), \Dv, -\St^{i}, -\Pv^{i}]^T.
\end{align}
The work of external forces $\delta W^{ext}$ shall be represented by the linear functional $\ell \in \spaceV^*$. Note that, for the present choice of independent unknowns $\uv$ and $\Dv$, the external work also contains boundary conditions on the electric potential. In the exemplary case of a body under mechanical volume load $\fv$, surface tractions $\tv$ on the boundary part $\Gamma_{trac}$, and a prescribed potential $\phi_0$ on the electrodes $\Gamma_{el}$, this functional is defined as
\begin{align}
\langle \ell, \zzz\rangle &:= \int_\Omega \fv \cdot \tilde \uv\, dx +  \int_{\Gamma_{trac}} \tv \cdot \tilde \uv\, ds - \int_{\Gamma_{el}} \phi_0 \, \tilde \Dv \cdot \nv\, ds.
\end{align}
Thus, the variational inequality \eqref{eq:varin_phys} translates to the abstract variational inequality of the form: Find $\uu: [0,T] \to \spaceX_0$ such that
\begin{align}
\langle A(\ww), \zzz - \dot \ww \rangle +  \Phi(\zzz) - \Phi(\dot \ww) &\geq \langle \ell, \zzz - \dot \ww \rangle &&\text{for all } \zzz \in \spaceV_0. \label{eq:varin_time}
\end{align}

The variational inequality can be extended to hold for all $\zzz \in \spaceV$, i.e. also for dielectric displacement updates with non-zero divergence. To this end, the dissipation function needs to be augmented by this restriction, see \cite{HanReddy:1999}. Indeed, the augmented dissipation maps all dielectric displacements with non-zero divergence to infinity. In the following, it shall be denoted by $j: \spaceV \to \Rbarplus := \mathbb{R} \cup \{+\infty\}$ and is defined as
\begin{align}
j(\zzz) := \Phi(\zzz) + \sup_{\phi \in L^2(\Omega)} \int_\Omega \opdiv \tilde\Dv \,  \phi\, dx. \label{eq:j}
\end{align}
Note that, when restricted to $\spaceV_0$, $j(\cdot)$ and $\Phi(\cdot)$ are equivalent. Therefore, we will use $j(\cdot)$ in the following, and consider the variational inequality (equivalent to \eqref{eq:varin_time}),
\begin{align}
\langle A(\ww), \zzz - \dot \ww \rangle +  j(\zzz) - j(\dot \ww) &\geq \langle \ell, \zzz - \dot \ww \rangle &&\text{for all } \zzz \in \spaceV. \label{eq:varin_time_j}
\end{align}

\subsection{Assumptions} \label{sec:assumptions}
We collect assumptions on the various functionals that we need to proof existence and uniqueness of the update in a time-discrete scheme. In Section~\ref{sec:application}, we show whether these assumptions are satisfied for some standard material models. The following definitions are taken from Brezis \cite{Brezis:1968}.

\begin{definition}\label{def:clc}
	Let $E$ be a Hilbert space. A  functional $\phi: E \to \Rbarplus$ is called \emph{convex} if and only if for all $\rho \in [0,1]$, $u_1, u_2 \in E$,
	\begin{align}
	\phi(\rho u_1 + (1-\rho) u_2) \leq \rho \phi(u_1) + (1- \rho) \phi(u_2).
	\end{align}
	The functional $\phi$ is called \emph{lower semicontinuous} if and only if for all $u \in E$, and all sequences $u_n \to u$ in $E$, 
	\begin{align}
	\liminf \phi(u_n) \geq \phi(u).\label{eq:lc}
	\end{align}
	It is \emph{weakly lower semicontinuos} if \eqref{eq:lc} holds for all weakly convergent sequences $u_n \rightharpoonup u$.
	
	An operator $A: X \to E$ is called hemicontinuous on a convex subset $X \subset E$ if for all $x, y \in X$ the mapping
	\begin{align}
	[0,1] \to \mathbb{R}, t \mapsto \langle A((1-t) x + t y), x-y\rangle
	\end{align}
	is continuous.
\end{definition}
As a minimal assumption we demand the Helmholtz free energy to be lower semicontinuous and convex. Moreover, its Frechet (or full) derivative shall exist and define the operator~$A$:
\begin{equation}
\begin{split}
\Psi:\spaceX \to \mathbb{R} &\text{ is convex, lower semicontinuous and}\\
&\text{Frechet differentiable with } A=D\Psi \label{eq:Psiconv} 
\end{split}
\end{equation}
To get convergence estimates and stability bounds, we further need that $A: \spaceX_0 \to \spaceV^*$ is strongly monotone, i.e. there exist $m > 0$ such that for all $\ww_1, \ww_2 \in \spaceX_0$,
\begin{equation}
\langle A(\ww_1) - A(\ww_2), \ww_1 - \ww_2 \rangle \geq m \|\ww_1 - \ww_2\|_{\spaceV}^2.
\label{eq:Acont}
\end{equation}

The dissipation function $\Phi: \spaceV \to \mathbb{R}$ as well as its augmented counterpart $j: \spaceV \to \Rbarplus$  need not be differentiable, but non-negative, proper, and positively homogeneous,
\begin{align} \label{eq:jproper}
j(\ww) \geq 0  \text{ for all } \ww,\zzz \in \spaceV,
&\text{ and it exists at least one }\ww\in \spaceV \text{ with } j(\ww) < \infty\text{, and}\\
j(\alpha \zzz) &= \alpha j(\zzz) \text{ for all } \alpha > 0, \zzz \in \spaceV.\label{eq:jposhom}
\end{align}
Additionally, we assume
\begin{align}\label{eq:jclc}
j: \spaceV \to \Rbarplus& \text{ is convex and lower semicontinuous, and}\\
j: \spaceV_0 \to \mathbb{R} & \text{ is continuous, i.e. } j(\zzz) \leq c \|\zzz\|_\spaceV \text{ for } \zzz \in \spaceV_0. \label{eq:jcont}
\end{align}

In Section~\ref{sec:timedep}, we aim at showing existence of a solution to the time-dependent variational inequality. To accomplish this task, we need an additional assumption on the energy, namely Lipschitz-continuity of $A$,
\begin{align} \label{eq:ALipschitz}
\langle A(\ww_1) - A(\ww_2), \zzz \rangle & \leq c_A \|\ww_1 - \ww_2\|_{\spaceV}\|\zzz\|_{\spaceV} & 
\text{ for all } \ww_1, \ww_2 \in \spaceX_0, \zzz \in \spaceV.
\end{align}

\section{Time discrete update equation} \label{sec:timediscr}

We use a uniform partitioning of the time interval $[0,T]$ into $N$ sub-intervals,
\begin{align}
0 = t_0 < t_1 < \dots < t_{N-1} < t_N = T \qquad \text{with} \qquad t_{n}-t_{n-1} = \Delta T = T/N.
\end{align}
For $N$ fixed, we will define a sequence $\{\ww_n\}_{n=0}^N \in [\spaceV_0]^{N+1}$ as consecutive solutions to (spatial but time-independent) variational inequalities. We use the backward difference $\Delta \ww_n = \ww_n - \ww_{n-1}$ and $\ell_n = \ell(t_n)$. We show that this sequence is defined uniquely, and that certain stability estimates are satisfied. We use the following existence result by Brezis:

\begin{theorem}[Corollaire~30 in \cite{Brezis:1968}] \label{brezis}
	Let $E$ be a reflexible Banach space, and let $X \subset E$ be closed and convex with $0 \in X$. Let $A: X \to E^*$ be weakly pseudo-monotone and $\phi: X \to ]-\infty, +\infty]$ be convex lower semicontinuous with $\phi(0) < \infty$. If 
	\begin{align}
	\lim_{\|x\|\to \infty} \frac{\langle A(x), x\rangle + \phi(x) }{\|x\|} = \infty, \label{eq:atoinf}
	\end{align}
	then for $\ell \in E^{*}$ there exists a solution $u \in X$ to
	\begin{align}
	\langle A(u), v-u \rangle + \phi(v) - \phi(u) &\geq \langle \ell, v - u\rangle && \forall v \in X.
	\end{align}
	If $A$ is additionally strongly monotone, the solution is unique.
\end{theorem}

The main result of this section is the following:
\begin{theorem} \label{theo:discrete}
	Let $\Psi: \spaceX \to \mathbb{R}$ be convex, lower semicontinuous and Frechet differentiable as in \eqref{eq:Psiconv}, and let $A: \spaceX \to \spaceV^*$ be its derivative. Let $j: \spaceV \to \Rbarplus$ be non-negative, proper, convex, positively homogeneous and lower semicontinuous as in \eqref{eq:jproper}, \eqref{eq:jposhom} and \eqref{eq:jclc}. 
	Let moreover $A: \spaceX_0 \to \spaceV^*$ be strongly monotone on $\spaceX_0$ as in \eqref{eq:Acont}.
	
	Then, for $N$ fixed and any given $\{\ell_n\}_{n=0}^N$ with $\ell_n \in \spaceV^*$ and $\ell_0 = 0$, there exists a unique sequence $\{\ww_n\}_{n=0}^N$ such that $\ww_n \in \spaceX_0$ and $\Delta \ww_n \in \spaceV_0$ and
	\begin{align}
	\langle A(\ww_n), \zzz - \Delta \ww_{n}) + j(\zzz) - j(\Delta \ww_n) \geq \langle \ell_n, \zzz - \Delta \ww_n\rangle \qquad \forall\ \zzz \in \spaceV. \label{eq:varindelta}
	\end{align}
	The set of test functions can be equivalently restricted to $\zzz \in \spaceV_0$. 
	%
	With the constant of monotonicity $m$ from \eqref{eq:Acont}, the solution satisfies the stability estimate
	\begin{align}
	\|\Delta \uu_n\|_\spaceV & \leq \frac{1}{m} \|\Delta \ell_n\|_{\spaceV^*}. \label{eq:firstest} 
	\end{align}
	
\end{theorem}
\begin{proof}
	We use Theorem~\ref{brezis} from convex analysis to show existence and uniqueness of the solutions. For the stability estimates, we progress along the lines of proof of \cite{HanReddy:1999} and see that some of their assumptions can be weakened. To show existence and uniqueness of the sequence $\{\uu_n\}_{n=0}^N$, we proceed inductively from $\ww_0 = \mathbf{0}$, assuming $\ww_{n-1}$ to be known. We rewrite the variational inequality \eqref{eq:varindelta} in terms of the unknown $\Delta \uu_n$ and $\uu_{n-1}$,
	\begin{align}
	\langle A(\Delta \uu_n + \uu_{n-1}), \zzz - \Delta \uu_n\rangle +  j(\zzz) - j(\Delta \uu_n)  \geq \langle \ell_n, \zzz - \Delta \uu_n\rangle \qquad \forall\ \zzz \in \spaceV_0. \label{eq:varindelta2}
	\end{align}
	
	We show that we can apply Theorem~\ref{brezis} to obtain existence of a solution $\Delta \uu_n$. As solution space we choose $E = \spaceV_0$. We see that admissible updates are in the set $\spaceX_{0,\uu_{n-1}} := \spaceX_0 - \uu_{n-1} = \{\zzz: \zzz + \uu_{n-1} \in \spaceX_0\}$ where saturation is not reached.
	From \cite[Proposition~1]{Brezis:1968} we know that convexity of $\Psi$ implies monotonicity and hemicontinuity of $A$ on the -- possibly open -- set $\spaceX_{0,\uu_{n-1}}$, which further implies pseudo-monotonicity of $A$. Obviously, strong monotonicity of $A$ and positivity of $j$ imply condition \eqref{eq:atoinf}. But still we cannot use  $\spaceX_{0,\uu_{n-1}}$ directly for $X$ in Theorem~\ref{brezis}, as this set is not necessarily closed. Instead, we use the parameter-dependent closed sub-set
	\begin{align}
	\spaceX_{0,\uu_{n-1}}(c) := \{ \zzz \in \spaceX_0: \langle A(\uu_{n-1} + \zzz), \zzz \rangle - \langle \ell_n, \zzz\rangle \leq c\|\zzz\| \}, \label{eq:x0c}
	\end{align}
	with the choice of $c > 0$ still to be determined.
	We show that this set is closed:  Assume a sequence $\zzz_k \in \spaceX_{0,\uu_{n-1}}(c)$ that converges strongly to some $\zzz \in \spaceV$. To show closedness of $\spaceX_{0,\uu_{n-1}}(c)$ we need to prove that $\zzz \in \spaceX_{0,\uu_{n-1}}(c)$. Starting from the defining condition of \eqref{eq:x0c} applied for $\zzz_k$,
	\begin{align}
	\langle A(\uu_{n-1} + \zzz_k), \zzz_k \rangle - \langle \ell_n, \zzz_k\rangle \leq c\|\zzz_k\|,
	\end{align}
	we apply $\liminf_{k\to \infty}$ on both sides. Proposition~6 in \cite{Brezis:1968} and continuity of $\ell_n$ ensure that
	\begin{align}
	\langle A(\uu_{n-1} + \zzz), \zzz \rangle - \langle \ell_n, \zzz\rangle &\stackrel{\text{\cite{Brezis:1968}} }{ \leq} \liminf_{k\to \infty} \langle A(\uu_{n-1} + \zzz_k), \zzz_k \rangle - \langle \ell_n, \zzz_k\rangle\\
	& \leq c \liminf_{k\to \infty} \|\zzz_k\| = c \|\zzz\|.
	\end{align}
	Thus we have shown $\zzz \in \spaceX_{0,\uu_{n-1}}(c)$ due to \eqref{eq:x0c}, and further that $\spaceX_{0,\uu_{n-1}}(c)$ is closed for any fixed $c > 0$.
	
	The functional $\phi = j$ satisfies the conditions of Theorem~\ref{brezis}.
	With all assumptions of Theorem~\ref{brezis} satisfied, there exists a solution $\Delta \uu_n^c$ to the parameter-dependent variational inequality
	\begin{align}
	\begin{split}
	\langle A(\Delta \uu_n^c + \uu_{n-1}), \zzz - \Delta \uu_n^c\rangle +  j(\zzz) - j(\Delta \uu_n^c)\ \geq\  &\langle \ell_n, \zzz - \Delta \uu_n^c\rangle \\& \forall\ \zzz \in \spaceX_{0,\uu_{n-1}}(c). 
	\end{split} \label{eq:varindelta_c}
	\end{align}
	As $A$ is strongly monotone, the solution is unique .
	
	It remains to be shown that the variational inequality holds also for test functions $\zzz \in \spaceV_0 \backslash \spaceX_{0,\uu_{n-1}}(c)$ provided $c$ is larger than some fixed value not depending on the solution. To preserve uniqueness, we have to show that no $\uu \in \spaceX_0 \backslash\spaceX_{0,\uu_n}(c)$ can be an additional solution. 
	
	For the first task, choose $c_{n-1} = c(\uu_{n-1})$ depending on the previous iterate such that
	\begin{align}
	\langle A(\uu_{n-1}), \zzz\rangle - \langle \ell_n, \zzz\rangle \leq c_{n-1} \|\zzz\| \qquad \forall \zzz \in \spaceV_0.
	\end{align}
	This constant exists since for $\uu_{n-1} \in \spaceX_0$ fixed, the derivative $A(\uu_{n-1})$ is a continuous linear operator. Now set $c = 2 c_{n-1}$, and let $\zzz \in \spaceX_0 \backslash \spaceX_{0,\uu_{n-1}}(2c_{n-1})$, i.e.
	\begin{align}
	\langle A(\uu_{n-1} + \zzz), \zzz \rangle - \langle \ell_n, \zzz\rangle > 2c_{n-1} \|\zzz\|.
	\end{align}
	From the hemicontinuity of $A$ we deduce that there exists some $t_0 > 0$ such that for all $t \in [0,t_0]$
	\begin{align}
	\langle A(\uu_{n-1} + t \zzz), t \zzz\rangle - \langle \ell_n, t \zzz\rangle \leq 2c_{n-1}\, t \|\zzz\| ,
	\end{align}
	i.e. $t \zzz \in \spaceX_{0,\uu_{n-1}}( 2c_{n-1})$. But then the variational inequality \eqref{eq:varindelta_c} is satisfied for $t \zzz$, and due to linearity and positive homogeneity of $j$, it is thus satisfied also for test function $\zzz \in \spaceV$.
	
	For the second task, we have to show that there cannot exist any $\Delta\uu \in \spaceX_0 \backslash\spaceX_{0,\uu_{n-1}}( 2c_{n-1})$ that is also solution to \eqref{eq:varindelta2}. But for such a function $\Delta\uu$ we know
	\begin{align}
	-\langle A(\uu_n + \Delta\uu), \Delta\uu\rangle + \langle \ell_n, \Delta\uu\rangle < -2c_{n-1}\|\uu\|  < 0.
	\end{align}
	Together with the positivity of dissipation, one immediately obtains that the variational inequality \eqref{eq:varindelta2} is not satisfied for $\zzz = 0$, namely it holds
	\begin{align}
	\langle A(\uu_n + \Delta\uu), -\Delta\uu\rangle - j(\Delta\uu) < - \langle \ell_n, \Delta\uu\rangle.
	\end{align}
	Thus $\Delta \uu_n$ is the only solution to \eqref{eq:varindelta2}.
	
	We proceed to the stability estimates \eqref{eq:firstest}. 
	We additionally assume strong monotonicity of $A$ as in \eqref{eq:Acont}. Note that in this case, condition \eqref{eq:atoinf} is trivially satisfied. In \eqref{eq:varindelta2}, we set $\zzz = \mathbf{0}$ to obtain
	\begin{align}
	\langle A(\Delta \uu_n + \uu_{n-1}), \Delta \uu_n\rangle + j(\Delta \uu_n) \leq \langle \ell_{n-1}+\Delta \ell_n, \Delta \uu_n\rangle. \label{eq:varin_z0}
	\end{align}
	Next, in \eqref{eq:varindelta} at time step $n-1$, we use $\zzz = \Delta \uu_n + \Delta \uu_{n-1} \in \spaceV_0$. By algebraic manipulations, using convexity and positive homogeneity of $j(\cdot)$, we see
	\begin{align}
	\langle A(\uu_{n-1}), \Delta \uu_n\rangle + \underbrace{j(\Delta \uu_n + \Delta \uu_{n-1}) - j(\Delta \uu_{n-1})}_{\leq j(\Delta \uu_n)} \geq \langle \ell_{n-1}, \Delta \uu_n\rangle. \label{eq:varinDeltawn1}
	\end{align}
	Subtracting \eqref{eq:varinDeltawn1} from \eqref{eq:varin_z0} and applying strong monotonicity on the one and continuity of $\Delta \ell_n$ on the other hand we arrive at the desired result \eqref{eq:firstest},
	\begin{align}
	m\|\Delta \uu_n\|_{\spaceV}^2 \leq \langle A(\Delta \uu_n + \uu_{n-1}) - A(\uu_{n-1}), \Delta \uu_n\rangle& \leq \langle \Delta \ell_n, \Delta \uu_n\rangle\\ &\leq \|\Delta \ell_n\|_{\spaceV^*} \|\Delta \uu_n\|_\spaceV.
	\end{align}
\end{proof}

\begin{lemma} \label{lemma:estwN}
	Assume that $\ell \in H^1(0,T; \spaceV^*)$ with $\ell(0) = 0$. Then the time-discrete solution $\{\ww_n\}_{n=0}^N$ from Theorem~\ref{theo:discrete} satisfies
	\begin{align}
	\max \|\ww_n\|_{\spaceV} &\leq c \|\dot\ell\|_{L^1(0,T;\spaceV)},\\
	\sum_{n=1}^N \frac{1}{\Delta T} \|\Delta \ww_n\|_\spaceV^2 & \leq c \|\dot\ell\|_{L^2(0,T;\spaceV)}^2. \label{eq:est_deltawn}
	\end{align}
\end{lemma}
\begin{proof}
	Follows directly from \cite[Lemma~7.2]{HanReddy:1999}.
\end{proof}

\section{Existence of a time-dependent solution} \label{sec:timedep}

We proceed to finding a time-dependent solution from series of time-discrete solutions. A similar approach can be found in the framework of elasto-plasticity \cite{HanReddy:1999}. We note that this approach is intrinsically different from the procedure used by Sofonea and Matei \cite{SofoneaMatei:2011} in contact mechanics. This approach \cite{SofoneaMatei:2011} uses viscosity to ensure existence and uniqueness, which is not present in our problem. 

We generate a time-dependent solution $\ww^N(t)$ from the series $\{\ww_n\}_{n=0}^N$ from the previous section. We do so by linear interpolation in time,
\begin{align} \label{eq:wN}
\ww^N(t) & := \ww_{n-1} + \frac{t - t_{n-1}}{\Delta T} \Delta \ww_n & &\text{for } t \in [t_{n-1}, t_n].
\end{align}
As, for any time $t$, $\ww^N(t)$ is a convex combination of $\ww_n \in \spaceX_0$, we observe that $\ww^N(t) \in \spaceX_0$.
The following lemma proves that this interpolated solution satisfies a modified time-dependent variational inequality for a certain kind of step functions $\zzz^N$.

\begin{lemma} \label{lemma:modifiedVI}
	Let all the assumptions from Theorem~\ref{theo:discrete} be satisfied, and
	let $\ww^{N}$ be defined from the series of solutions as in \eqref{eq:wN}. Let $\ell^N$ denote the corresponding piecewise linear time interpolant of the time-dependent right hand side $\ell \in H^1(0,T; \spaceV^*)$.
	For any sequence $\{\zzz_n\}_{n=1}^N$ in $\spaceV_0$, let $\zzz^N(t)$ be defined piecewise by
	\begin{align} \label{eq:zN}
	\zzz^N(t) & := \zzz_n & &\text{for } t \in [t_{n-1}, t_n].
	\end{align}
	Let $A: \spaceX_0 \to \spaceV^*$ be additionally Lipschitz continuous as in \eqref{eq:ALipschitz}. Then, there exists some constant $c>0$ such that the following inequality is satisfied,
	\begin{align}
	\int_0^T \Big( \langle A(\ww^N), \zzz^N - \dot \ww^N\rangle + &j(\zzz^N) - j(\dot \ww^N) - \langle \ell^N, \zzz^N - \dot \ww^N\rangle \Big) \geq \\
	&\geq
	-c \Delta T ( \|\dot \ell\|_{L^2(0,T;\spaceV)} \|\zzz^N\|_{L^2(0,T;\spaceV)} +  \|\dot \ell\|_{L^2(0,T;\spaceV)} ^2 ) .
	\end{align}
	
\end{lemma}
\begin{proof}
	For each time step $1\leq n \leq N$, the variational inequality \eqref{eq:varindelta} is satisfied, if we choose test function $\zzz = \Delta T/2(\zzz_n + \zzz_{n+1})$ with $\zzz_{N+1} = 0$, and use the positive homogeneity of $j$,
	\begin{align}
	\left\langle A(\ww_n), \Delta T\frac{\zzz_n + \zzz_{n+1}}{2} - \Delta \ww_{n}\right\rangle + &\frac{\Delta T}{2} j(\zzz_n + \zzz_{n+1}) - j(\Delta \ww_n) \geq \nonumber\\
	&\geq \left\langle \ell_n, \Delta T\frac{\zzz_n + \zzz_{n+1}}{2} - \Delta \ww_n\right\rangle. \label{eq:varindelta_n}
	\end{align}
	Summing over $n$ leads to
	\begin{align}  \label{eq:firstsum}
	\sum_{n=1}^N & \Delta T\left(\left\langle A(\ww_n), \frac{\zzz_n + \zzz_{n+1}}{2} - \frac{\Delta \ww_{n}}{\Delta T}\right\rangle + \frac{\Delta T}{2} j(\zzz_n + \zzz_{n+1}) - j(\Delta \ww_n)\right)\\
	& \geq  \sum_{n=1}^N \Delta T\left\langle \ell_n, \frac{\zzz_n + \zzz_{n+1}}{2} - \frac{\Delta \ww_n}{\Delta T}\right\rangle.
	\end{align}
	In \cite{HanReddy:1999} the following estimates have been shown,
	\begin{align}
	\sum_{n=1}^N \frac{\Delta T}{2} j(\zzz_n + \zzz_{n+1}) &\leq \int_0^T j(\zzz^N(s)) ds - \frac{\Delta T}{2} j(\zzz_1),\\
	\sum_{n=1}^N  j(\Delta \ww_n) &= \int_0^T j(\dot \ww^N(s)) ds,\\
	\sum_{n=1}^N \Delta T\left\langle \ell_n, \frac{\zzz_n + \zzz_{n+1}}{2} - \frac{\Delta \ww_n}{\Delta T}\right\rangle & \geq
	\int_0^T \langle \ell^N(s), \zzz^N(s) - \dot \ww^N(s)\rangle ds + c \Delta T \int_0^T \|\dot \ell(s)\|_{\spaceV^*}^2\, ds.
	\end{align}
	Using the Lipschitz continuity of $A$, we see that the first part of the first sum in \eqref{eq:firstsum} is close to a corresponding integral. We estimate
	\begin{align}
	&\left| \sum_{n=1}^N  \Delta T \left\langle A(\ww_n), \frac{\zzz_n + \zzz_{n+1}}{2} \right\rangle - \int_0^T \langle A(\ww^N(s)), \zzz^N(s)\rangle\, ds \right| \\
	&=\left| \sum_{n=1}^N \left(\! \frac{\Delta T}{2} \left\langle A(\ww_n)+ A(\ww_{n-1}), \zzz_n \right\rangle - \int_0^{\Delta T}\!\!\!\!\!\! \langle A(\ww_{n-1}+s\,\Delta\ww_n/\Delta T), \zzz_n\rangle\, ds\right) \right|\\
	&= \frac12 \left| \sum_{n=1}^N \int_0^{\Delta T}   \left\langle A(\ww_n)+A(\ww_{n-1}) - 2A(\ww_{n-1}+s\,\Delta\ww_n/\Delta T) , \zzz_n \right\rangle \, ds \right|\\
	&\leq  \frac12\sum_{n=1}^N \int_0^{\Delta T} 2c_A \frac{s}{\Delta T} \|\Delta \ww_n\|_\spaceV \|\zzz_n\|_\spaceV \, ds\\
	&=  \frac12\sum_{n=1}^N c_A \Delta T \|\Delta \ww_n\|_\spaceV \|\zzz_n\|_\spaceV.
	\end{align}
	Using Cauchy-Schwarz inequality and estimate \eqref{eq:est_deltawn}, we further derive
	\begin{align}
	\left| \sum_{n=1}^N  \Delta T \left\langle A(\ww_n), \frac{\zzz_n + \zzz_{n+1}}{2} \right\rangle - \int_0^T \langle A(\ww^N(s)), \zzz^N(s)\rangle\, ds \right|&\\
	\leq c \Delta T \|\dot \ell\|_{L^2(0,T; \spaceV)} \|\zzz^N\|_{L^2(0,T;\spaceV)}&.
	\end{align}
	By similar conclusions we estimate the difference between the second part of the first sum in \eqref{eq:firstsum} and the corresponding integral,
	\begin{align}
	&\left| \sum_{n=1}^N  \left\langle A(\ww_n),\Delta \ww_n \right\rangle - \int_0^T \langle A(\ww^N(s)), \dot \ww^N(s)\rangle\, ds \right|\\
	&= \left| \sum_{n=1}^N  \left( \left\langle A(\ww_{n-1}+\Delta \ww_n),\Delta \ww_n \right\rangle - \int_0^{\Delta T} \left\langle A(\ww_{n-1} + s \frac{\Delta \ww_n}{\Delta T}), \frac{\Delta \ww_n}{\Delta T} \right\rangle\, ds \right) \right|\\
	&=  \left| \sum_{n=1}^N \frac{1}{\Delta T} \int_0^{\Delta T} \left\langle A(\ww_{n-1}+\Delta \ww_n) - A(\ww_{n-1} + s \frac{\Delta \ww_n}{\Delta T}), \Delta \ww_n \right\rangle \, ds  \right|\\
	&\leq \sum_{n=1}^N \frac{c_A}{\Delta T}\int_0^{\Delta T} \left(1-\frac{s}{\Delta T}\right)\, ds\, \|\Delta \ww_n\|_{\spaceV}^2\\
	&= \frac{c_A}{2}\sum_{n=1}^N\|\Delta \ww_n\|_{\spaceV}^2 \ \leq c \Delta T \|\dot \ell\|_{L^2(0,T; \spaceV)}^2.
	\end{align}
	Putting these estimates together, and observing the positivity of $j(\zzz_1)$, we arrive at the desired inequality.
\end{proof}

Next, we provide a candidate for the time-dependent solution of the original variational inequality. We observe that $\ww^N$ is bounded in the sense that, for any $N$,
\begin{align}
\|\ww^N\|_{L^\infty(0,T; \spaceV)} &\leq c, & \|\dot \ww^N\|_{L^2(0,T; \spaceV)} &\leq c_A.
\end{align} 
This follows directly from Lemma~\ref{lemma:estwN}.  Additionally, from the Lipschitz continuity of $A$ we know that
\begin{align}\label{eq:convAwN}
\|A(\ww^N)\|_{L^\infty(0,T; \spaceV)} &\leq c.
\end{align} 
Now, consider a fixed step size $N_0$ and the according sequence of step sizes $N_l = 2^{-l} N_0$. To this sequence, there exists a weakly convergent subsequence, without loss of generality again denoted by $N$ such that
\begin{align}
\ww^N & \weakstarconv \ww \text{ in } L^\infty(0,T;\spaceV), & \dot \ww^N & \weakconv \dot \ww \text{ in } L^2(0,T;\spaceV), & A(\ww^N) & \weakconv A(\ww) \text{ in } L^2(0,T;\spaceV^*).
\end{align}
Applying $\limsup$ and $\liminf$ to the left and right hand side of the variational inequality from Lemma~\ref{lemma:modifiedVI}, we know that (with $\Delta T = T/N$),
\begin{align}
\limsup_{N \to \infty} \int_0^T &\Big( \langle A(\ww^N), \zzz^{N_0} - \dot \ww^N\rangle + j(\zzz^{N_0}) - j(\dot \ww^N) - \langle \ell^N, \zzz^{N_0} - \dot \ww^N\rangle \Big) \geq \\
&\geq
\liminf_{N \to \infty} \Big( -c \Delta T ( \|\dot \ell\|_{L^2(0,T;\spaceV)} \|\zzz^{N_0}\|_{L^2(0,T;\spaceV)} +  \|\dot \ell\|_{L^2(0,T;\spaceV)} ^2 )  \Big).
\end{align}
Obviously, the limit of the right hand side is zero, as $\Delta T = T/N \to 0$. In \cite{HanReddy:1999} it is shown that
\begin{align}
&\limsup_{N \to \infty} \int_0^T- j(\dot \ww^N)dt  = -\liminf_{N \to \infty} \int_0^T j(\dot \ww^N)dt  \leq \int_0^T j(\dot \ww)dt,\\
&\lim_{N \to \infty} \int_0^T \langle \ell^N, \zzz^{N_0} - \dot \ww^N\rangle dt = \int_0^T \langle \ell, \zzz^{N_0} - \dot \ww\rangle dt 
\end{align}
Due to the weak convergence of $A(\ww^N)$ \eqref{eq:convAwN}, we see
\begin{align}
\limsup_{N \to \infty} \int_0^T \langle A(\ww^N), \zzz^{N_0}\rangle dt = \int_0^T \langle A(\ww), \zzz^{N_0}\rangle dt.
\end{align}
Last, we observe from the chain rule of differentiation, the weak lower semicontinuity of the energy $\Psi$ and the weak-star convergence of $\ww^N$ in $L^\infty(0,T; \spaceV)$ that
\begin{align}
\limsup_{N \to \infty} \int_0^T -\langle A(\ww^N), \dot \ww^N\rangle dt &= -\liminf_{N \to \infty} \int_0^T \frac{d}{dt} \Psi(\ww^N)\, dt \\
&= -\liminf_{N \to \infty} \Psi(\ww^N(T))\\ & \leq  -\Psi(\ww(T)) = - \int_0^T \langle A(\ww), \dot \ww\rangle dt.
\end{align}
Collecting these results, we find that $\ww$ satisfies the following variational inequality
\begin{align}
\int_0^T &\Big( \langle A(\ww), \zzz^{N_0} - \dot \ww\rangle + j(\zzz^{N_0}) - j(\dot \ww) - \langle \ell, \zzz^{N_0} - \dot \ww\rangle \Big) \geq 
0.
\end{align}
Approximating $\zzz \in L(0,T;\spaceV_0)$ by step functions and using a localization argument in time, the weak limit $\ww$ can be shown to satisfy the time-dependent variational inequality \eqref{eq:varin_time_j}. For details on this procedure, we refer to \cite[p. 165]{HanReddy:1999}. Thereby, we arrive at the desired result:

\begin{theorem}
	There exists a solution $\ww \in H^1(0,T; \spaceV)$ solving the time-dependent variational inequality
	\begin{align}
	\langle A(\ww), \zzz - \dot \ww \rangle +  j(\zzz) - j(\dot \ww) &\geq \langle \ell, \zzz - \dot \ww \rangle &&\text{for all } \zzz \in \spaceV. 
	\end{align}
\end{theorem}
\begin{proof}
	We have seen that $\ww$ above satisfies the variational inequality, and that $\ww \in L^\infty(0,T; \spaceV)$ and $\dot \ww \in L^2(0,T; \spaceV)$. From the Sobolev embedding theorem we deduce that then $\ww \in H^1(0,T; \spaceV)$.
\end{proof}

\section{Application to different ferroelectric material models} \label{sec:application}

In the following, we motivate in how far the assumptions of the previous sections are reasonable, and if they hold for some standard material models.  We do so first for a simple model without saturation, and for the ferroelectric material model for non-remanent straining as proposed by Landis \cite{Landis:2002}.  We derive the form of the dissipation function $\Phi(\dot \Pv^{i})$ for a given switching surface depending on $\hat \Ev$.
We will see that the assumptions for existence of an update solution (Theorem~\ref{theo:discrete}) are met in both cases if material parameters are in common ranges, but that Lipschitz continuity lacks for the latter model, and thereby convergence in time is not guaranteed by our deductions. 

In all cases, we assume the reversible part of the energy as proposed by \cite{Landis:2002},
\begin{align}
\begin{split}
\Psi^r := \int_\Omega \Big(\frac12 (\St - \St^{i}) : \cten : (\St - \St^{i}) - &(\St - \St^{i}) : \hten \cdot (\Dv - \Pv^{i})\\
& + \frac12 (\Dv - \Pv^{i}) \cdot \betat \cdot (\Dv - \Pv^{i}) \Big) \,dx .
\end{split} \label{eq:PsiR}
\end{align}
In the theoretical considerations below, we assume the case of non-remanent straining, i.e.\ $\St^{i} = 0$  and $\cten, \hten$ and $\betat$ depend on the remanent polarization.
Quite similarly, one might introduce a kinematic assumption for the remanent strain $\St^{i} = \St^{i}(\Pv^{i})$, as done in \cite{McMeekingLandis:2002,MieheRosatoKiefer:2011}. We consider this case in our numerical examples. The different material models will be characterized via different irreversible energies.

We assume that the material constants are such that the compound material tensor has positive eigenvalues bounded away from zero independently of $\Pv^{i}$. This issue is treated in detail by Stark et al. \cite{StarkNeumeisterBalke:2016a}, they give suitable conditions. However, they claim that these conditions are not met by all data-sets provided for commercially available ferroelectric materials.
Considering the special form of $\Psi^{r}$, in case these conditions are met, the derivative $A^{r} := \partial \Psi^{r}/\partial \uu$ is not strictly monotone, but satisfies
\begin{align} \label{eq:Arstrictlymonotone}
\langle A^r(\uu) - A^{r}(\zzz), \uu - \zzz \rangle \geq
c_1 (\|\St - \tilde \St\|_{L^2}^2 + \|\Dv - \Pv^{i} - \tilde \Dv + \tilde \Pv^{i}\|_{L^2}^2)
\end{align}

Before we turn to the different material models, we first characterize the dissipation function. Usually, it is defined by a threshold or switching surface, i.e. a condition on the  forces $\hat \Ev$ and $\hat \Tt$. We assume a condition of the form 
\begin{align}
\varphi(\hat \Ev) = |\hat \Ev|/E_0 - 1 \leq 0.
\end{align}
Then the dissipation function $\Phi$ is given by\textit{}
\begin{align}
\Phi(\dot \Pv^{i}) = \sup \Big\{ \int_\Omega \dot \Pv \cdot \hat \Ev\, dx:\ \hat \Ev \in [L^2]^3, \varphi(\hat \Ev) \leq 0\Big\} = \int_\Omega E_0 |\dot \Pv^{i}|\, dx.
\end{align}
In \cite{HanReddy:1999} it is shown that dissipation functions of the above format satisfy all assumptions such as convexity, lower semicontinuity and positive homogeneity. Indeed, in case of elasto-plasticity the dissipation function is of the same abstract form, where the coercive electric field $E_0$ resembles the yield stress.

\subsection{A simple ferroelectric material model} \label{sec:simplemodel}

A very simple material model in the spirit of Kamlah \cite{Kamlah:2001} assumes a quadratic dependence of the irreversible energy on the remanent polarization $\Pv^{i}$ via $\Psi^{i} = \int_\Omega H_0 \Pv^{i} \cdot \Pv^{i}\, dx$. Saturation is not included in this model. In this case, the solution space is the whole space, $\spaceX = \spaceV$.
As a quadratic form, obviously $\Psi^{i}$ is convex and Lipschitz continuous on $\spaceV$, and also
\begin{align} \label{eq:Aistrictlymonotone}
\langle A^i(\uu) - A^{i}(\zzz), \uu - \zzz \rangle = H_0 \|\Pv^{i} - \tilde \Pv^{i}\|_{L^2}^2
\end{align}
in the notation of the previous sections.

We can now prove strict monotonicity of $A = A^{r} + A^{i}$. One immediately deduces by \eqref{eq:Arstrictlymonotone} and \eqref{eq:Aistrictlymonotone} that
\begin{align} \label{eq:monotonicity}
\langle A(\uu) - A(\zzz), \uu - \zzz \rangle \geq \frac12 \min(c_1,H_0)(\|\St - \tilde \St\|_{L^2}^2 + \|\Dv - \tilde \Dv\|_{L^2}^2 + \|\Pv^{i} - \tilde \Pv^{i}\|_{L^2}^2).
\end{align}
The last, essential ingredient to strict monotonicity is the fact that we restricted the full space of all dielectric displacements $\spaceD$ to those which are  divergence-free, $\opdiv \Dv = 0$, and that this space $\spaceD_0$ is closed. In this case, the $L^2$ norm is equivalent to the full $H(\opdiv)$ norm on $\spaceD_0$, and \eqref{eq:monotonicity} is sufficient for strict monotonicity of $A$ on $\spaceV_0$.

\subsection{A saturating ferroelectric material model}

For the second material model, the irreversible part of the energy is assumed such that its derivative reads 
\begin{align}
\frac{\partial \Psi^{i}}{\partial \Pv^{i}} &= \frac{H_0 P_0^m}{2(m-1)}\left((P_0-|\Pv^{i}|)^{1-m}-(P_0+|\Pv^{i}|)^{1-m} \right) \frac{\Pv^{i}}{|\Pv^{i}|}.
\label{eq:PsiI1}
\end{align}
In \cite{Landis:2002} it has been shown that a free energy of similar form is differentiable, and all differentiations are provided analytically. 
The irreversible part $A^{i} := \partial \Psi^{i}/\partial \uu$ is strictly monotone in the remanent polarization,
\begin{align}
\langle A^i(\uu) - A^{i}(\zzz), \uu - \zzz \rangle \geq H_0 \|\Pv^{i} - \tilde \Pv^{i}\|_{L^2}^2.
\end{align}
Thus, strict monotonicity of $A = A^{r} + A^{i}$ follows the same way as in Section~\ref{sec:simplemodel}. However, $A^{i}$ is not Lipschitz continuous, as
\begin{align}
\langle A^i(\uu_1), \zzz \rangle \to  \infty &&\text{ as } \uu_1 \text{ approaches saturation.}
\end{align}
Although we can prove solvability of the time-discrete update equation, existence of a time-dependent solution is not guaranteed by our line of proof. However, in numerical examples, we did not meet any convergence problems. Concerning the introduction of polarization strains $\St^{i}(\Pv^{i})$, we refer to \cite{BotteroIdiart:2018}, where convexity of the potential was analyzed.

\section{Finite element implementation} \label{sec:fe}

\subsection{Finite element spaces} \label{sec:fespaces}
We propose to use conforming finite element spaces for the discretization of the variational inequality. For a simplicial finite element mesh $\mathcal{T} = \{T\}$ and $k\geq 0$, we use the nodal space of order $k+1$ for the displacements, the Brezzi-Douglas-Marini space $\mathcal{BDM}_{k+1}$ for the divergence-conforming dielectric displacement (see e.g. \cite{BoffiBrezziFortin:2013}), and piecewise defined remanent polarizations of order $k$,
\begin{align}
\uv \in \spaceU_h &:= \{\uv \in [H^1(\Omega)]^d: \uv|_T \in [P^{k+1}(T)]^d, \uv = 0 \text{ on } \Gamma_{fix}\} \subset \spaceU,\\
\Dv \in \spaceD_h &:= \mathcal{BDM}_{k+1} \subset \spaceD,\\
\Pv^{i} \in \spaceP &:= \{\Pv \in [L^2(\Omega)]^d: \Pv|_T \in [P^{k}(T)]^d\} \subset \spaceP.
\end{align}
Gauss' law of divergence free dielectric displacements is enforced by a Lagrangian multiplier in the sense of \eqref{eq:j}, which coincides with the electric potential $\phi$ and is discretized also by piecewise order $k$ functions,
\begin{align}
\phi \in \spaceW_h &:= \{\phi \in L^2(\Omega): \phi|_T \in P^{k}(T)\}.
\end{align}
Note that it is essential to choose dielectric displacement and its Lagrangian multiplier in a stable combination of spaces, such that not only
\begin{align}
\int_\Omega \opdiv \Dv\, \phi \, dx = 0 \text{ for all }\phi \in \spaceW_h \text{ implies } \opdiv \Dv = 0,
\end{align}
but also the discrete inf-sup condition holds independently of the mesh size,
\begin{align}
\inf_{\Dv \in \spaceD_h} \sup_{\phi \in \spaceW_h} \frac{\int_\Omega \opdiv \Dv\, \phi \, dx}{\|\Dv\|_{H(\opdiv)} \|\phi\|_{L^2}} \geq c.
\end{align}
This condition is satisfied for the pair $\mathcal{BDM}_{k+1}$ and piecewise order $k$ functions. For a thorough theoretical background we refer the interested reader to the exhaustive monograph \cite{BoffiBrezziFortin:2013} on mixed problems. We mention that other stable choices of finite element pairs exist. In our numerical results, we used a subspace of divergence-free $\mathcal{BDM}_{k+1}$ elements. Then the electric potential is discretized using only one degree of freedom per element, regardless of the approximation order of $\Dv$. In this case, it is impossible to evaluate the electric field as a derivative of $\varphi$. In any case, we recommend to use the constitutive law \eqref{eq:et}, as this leads to more accurate results.

\subsection{Regularization of the dissipation function} \label{sec:reg}

For solving variational inequalities, various numerical algorithms are proposed in the literature. Well-known for dual variational inequalities is the \emph{return-mapping algorithm} in different variants. There, after a predictor step, the generalized stress is projected back to the admissible set in the corrector step. Also for primal variational inequalities, as derived in this work, predictor/corrector iterations have been analyzed e.g. in the application of elasto-plasticity \cite[Section 12.2]{HanReddy:1999}. In all these methods, after an ``reversible'', i.e. linear, predictor step the remanent quantities are altered in a consistent way in the corrector step.

In contrast,  we propose to regularize the non-differentiable dissipation function, such that the problem can be solved ``all at once'' in a single Newton iteration. This \emph{regularization technique} has been analyzed for convergence and accuracy in \cite[Section~12.4]{HanReddy:1999}. Briefly, for a given regularization parameter $\varepsilon$, the non-differentiable dissipation $j$ is replaced by a smooth function $j_\varepsilon$, which differs from $j$ only by $\varepsilon$ (see Section~\ref{sec:numerics} for a special choice). If the findings from \cite{HanReddy:1999} can be transferred to the ferroelectric polarization problem, one can expect that for  the regularization parameter $\varepsilon$,
\begin{itemize}
	\item the solution $\uu_\varepsilon$ converges to $\uu$ in $\spaceV$ and
	\item $\|\uu - \uu_\varepsilon\|_\spaceV \leq c \sqrt{\varepsilon}$.
\end{itemize}
However, we do not aim at proving these convergence estimates for the present problem. \oldnew{}{We do not expect faster convergence as compared to return mapping algorithms with correct tangential stiffnesses. Indeed, iteration counts presented in Section~\ref{sec:cantilever} suggest a similar behavior. From our point of view, the main benefit of our approach lies in the much simpler implementation of the regularized dissipation function, as is described in Section~\ref{sec:impl}. Also, the possibility using higher order finite elements is given directly.}

\subsection{Implementation in Netgen/NGSolve} \label{sec:impl}

We use the software package Netgen/NGSolve available at {\tt https://ngsolve.org}. Netgen/NGSolve is an all-purpose finite element code, where high-order hierarchical finite elements for all element types (segments, triangles, quadrilaterals, tetraderda, hexahedra, prisms, \dots) and many different spaces (continuous or discontinuous, $\opcurl$ or $\opdiv$ conforming, \dots) are implemented. Via a Python interface, variational equations or even energy formulations can be entered symbolically. The (symbolic) equations are differentiated automatically, and a Newton iteration can be realized in a straightforward manner, without need to implement tedious tangent stiffnesses etc. by hand for each formulation.

In the present manuscript, an energy formulation was used, where the free energy was entered analytically, and a regularized version of the dissipation was added,
\begin{align}
j_\varepsilon(\dot \Pv^{i}) &:= \int_\Omega E_0 |\dot \Pv^{i}|_\varepsilon\, dx \text{ with } \\
|\dot \Pv^{i}|_\varepsilon &:= 
\left\{ \begin{array}{ll} 
|\Pv^{i}| - \varepsilon/2 & \text{if } |\Pv^{i}| \geq \varepsilon,\\
1/(2\varepsilon) |\Pv^{i}|^2 & \text{else.} \end{array} \right.
\end{align}

\section{Numerical results} \label{sec:numerics}

We provide a patch test example, where we reproduce known hysteresis effects and mechanical depolarization on a ferroelectric cube. In the second example, a ferroelectric cantilever is polarized by an applied electric field, and partially depolarized in bending. In both examples, we use the energies $\Psi^{r}$ and $\Psi^{i}$ as described in \eqref{eq:PsiR} and \eqref{eq:PsiI1}. We assumed the permittivity at constant strain $\betat = \betat^{S} = \epsilon^{-1} \It$ to  be independent of $\Pv^{i}$. The stiffness at constant electric field $\cten^{E}$ shall also be isotropic and independent of $\Pv^{i}$, and is characterized by Young's modulus $E_Y$ and Poisson ratio $\nu$. The piezoelectric tensor $\dten$ depends on $\Pv^{i}$ in the standard way as given in \cite[eq.~(4.3)]{Landis:2002}. Then the coupling tensor $\hten$ and the stiffness at constant dielectric displacement can be computed algebraically from $\betat^S$, $\dten$ and $\cten^{E}$, compare \cite{Landis:2002}. 

We include remanent straining, where $\St^{i}$ is assumed to depend directly on the polarization. We use the following formula provided in \cite{McMeekingLandis:2002},
\begin{align}
\St^{i}(\Pv^{i}) &:= \frac{S_0}{2 P_0} (3 \Pv^{i} \otimes \Pv^{i} - |\Pv^{i}|^2 \It).
\end{align}
\oldnew{}{This model is capable of mechanic depolarization under pressure, as is shown in \cite{McMeekingLandis:2002}, and is reproduced in Section~\ref{sec:patchtest}. It does not include purely elastic remanent straining in absence of polarization, though. To this end, an independent
	polarization strain as proposed by Landis \cite{Landis:2002} needs to be added. }

\subsection{Patch test} \label{sec:patchtest}

Consider a cube of side length $\SI{2}{\milli\meter}$ whose normal displacement is fixed at the three coordinate planes. The cube is electroded on top and bottom, the other faces are electrically insulated. We use material constants derived from the dimensionless constants proposed by \cite{McMeekingLandis:2002}. We set $E_0 = \SI{1000}{\volt\per\milli\meter}$, $P_0 = \SI{0.3}{\coulomb\per\meter\squared}$, $S_0 = 0.002$, $m=2$, $\epsilon = \SI{1.2e-8}{\coulomb\per\volt\per\meter}$, $E_Y = \SI{3e10}{\newton\per\meter\squared}$, $\nu = 0.3$, $d_{31} = \SI{-2.1e-10}{\meter\per\volt}$, $d_{33} = \SI{4.2e-10}{\meter\per\volt}$, $H_0 = \tfrac13\times10^{6}\, \si{\volt\meter\per\coulomb}$.  The regularization parameter from Section~\ref{sec:impl} is set to $\varepsilon = P_0 \times 10^{-6}$.

We provide hysteresis curves for the standard load case of electric polarization and depolarization by an electric field of $1.5 E_0$ in Figure~\ref{fig:hysteresis}. Moreover, Figure~\ref{fig:depol} shows the effect of mechanical depolarization by a compressive load of $\SI{200}{\newton\per\milli\meter\squared}$ applied to the top surface of the cube.

\begin{figure}
	\begin{center}
		\includegraphics[width=0.4\textwidth]{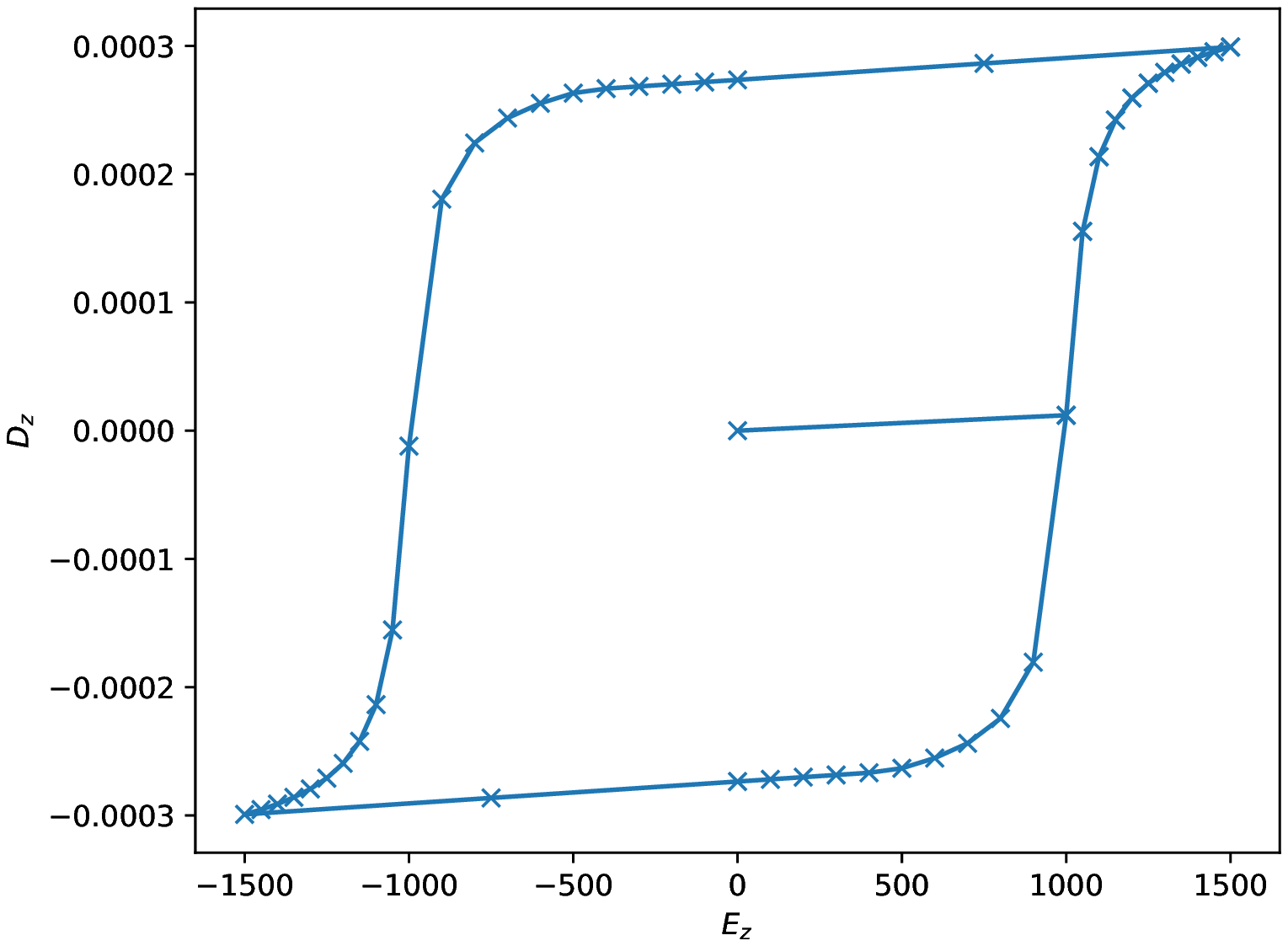}
		\includegraphics[width=0.4\textwidth]{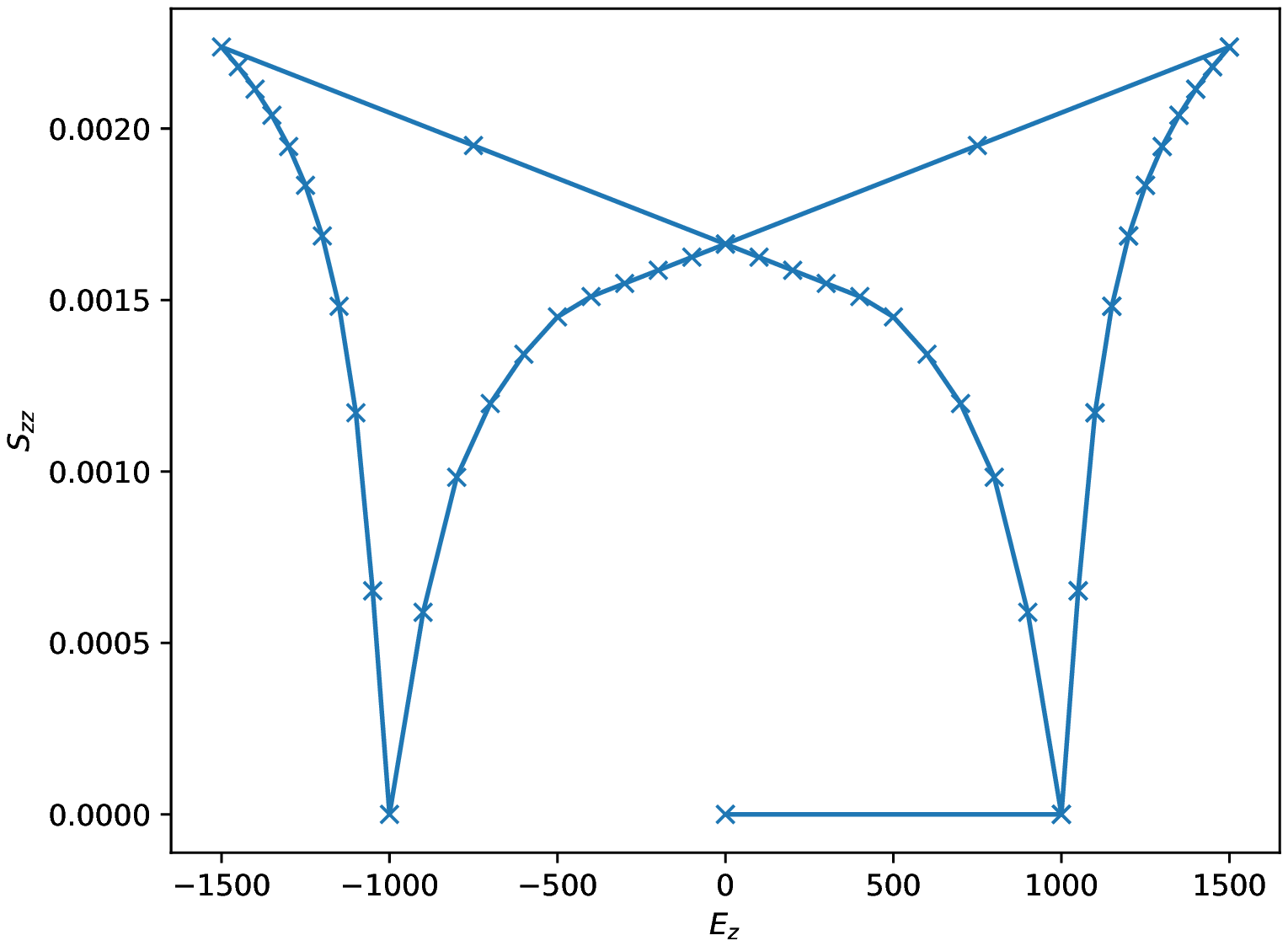}
	\end{center}
	\caption{Loadcase electric polarization and depolarization: Hystereses of dielectic displacement (left) and strain (right).} \label{fig:hysteresis}
\end{figure}
\begin{figure}
	\begin{center}
		\includegraphics[width=0.4\textwidth]{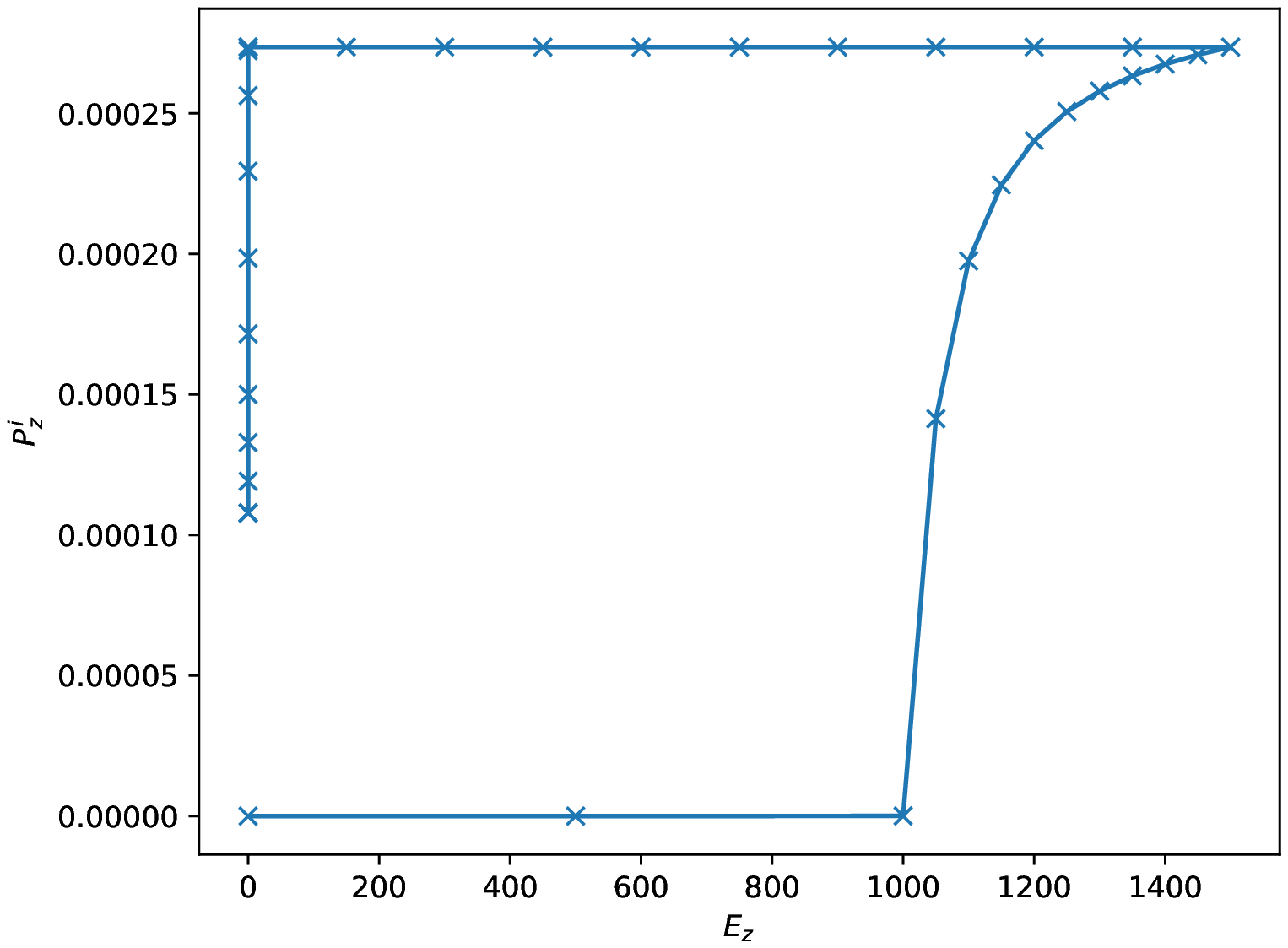}\\
		\includegraphics[width=0.4\textwidth]{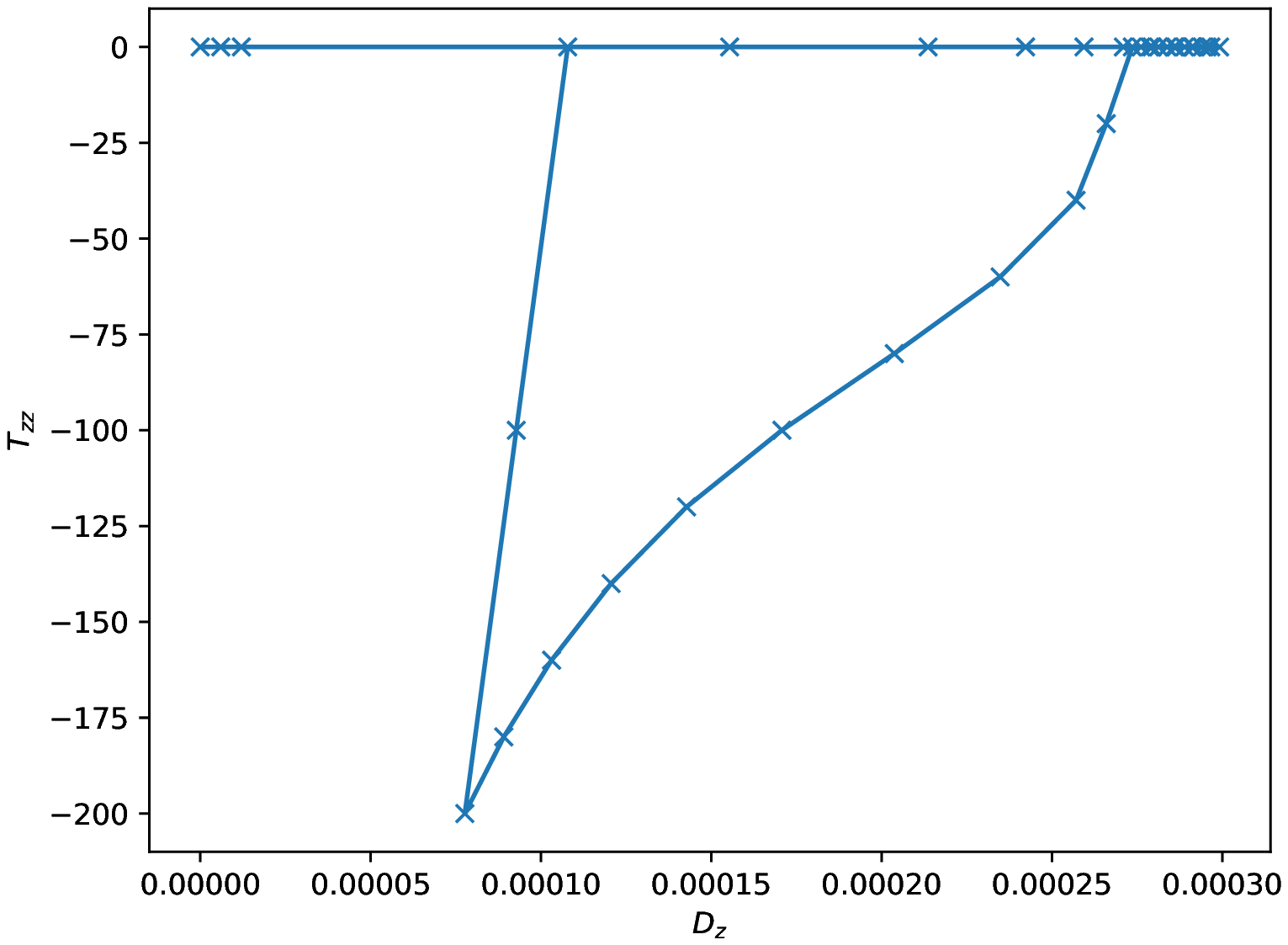}
		\includegraphics[width=0.4\textwidth]{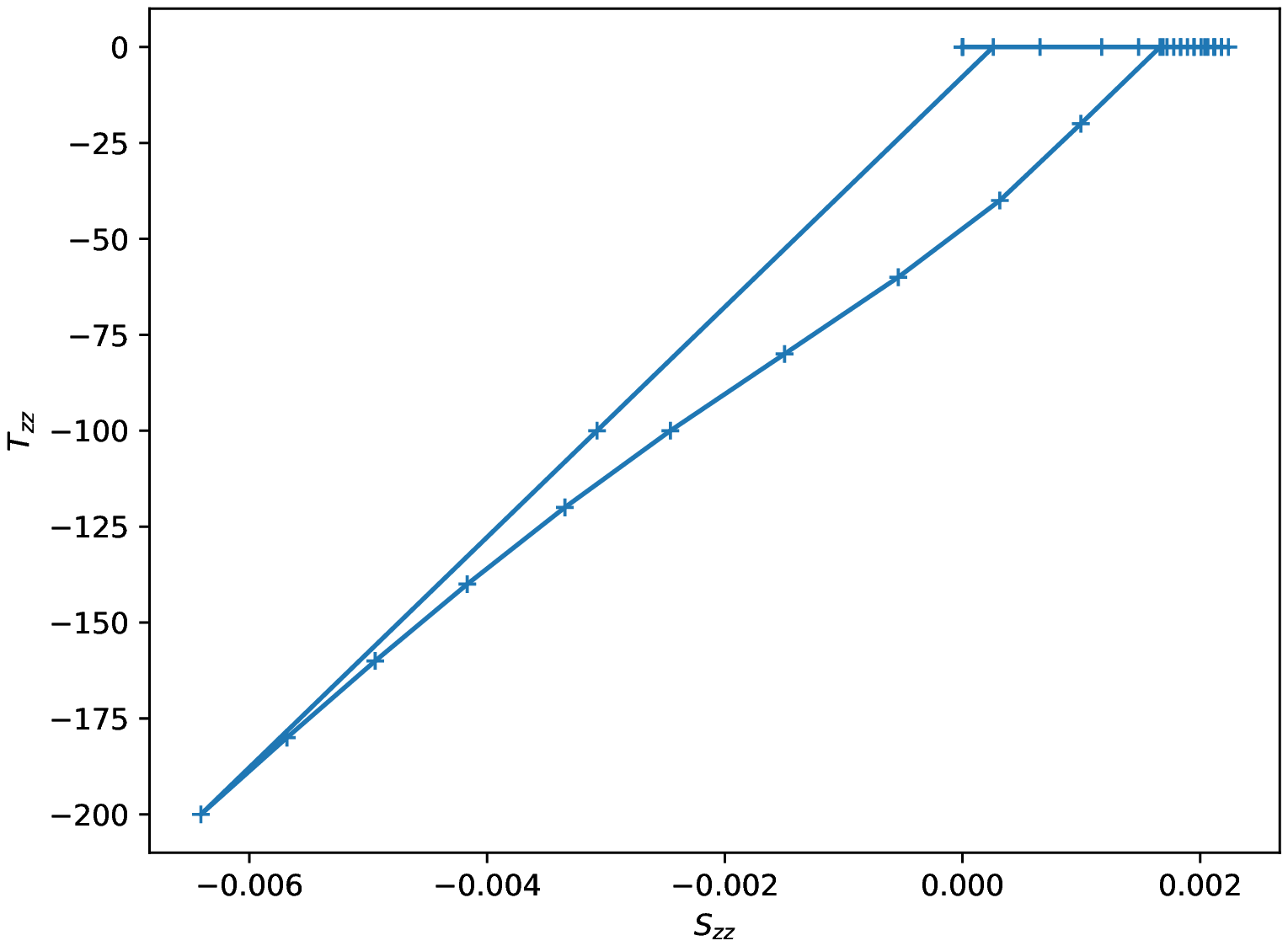}
	\end{center}
	\caption{Loadcase mechanic depolarization: Remanent polarization as a function of the applied electric field (top), stress evolution over dielectric displacement (lower left) and over strain (lower right).} \label{fig:depol}
\end{figure}

\subsection{Ferroelectric cantilever} \label{sec:cantilever}

The second example is that of a ferroelectric cantilever beam of length $\SI{2}{\milli\meter}$ and cross section $2\times2\, \si{\milli\meter}$, which was proposed by \cite{ZouariZinebBenjeddou:2011}. The clamped end as well as the tip of the beam are electroded. In the first loading cycle, the beam is polarized applying an electric potential to the beam tip while keeping the other electrode grounded. The electric field is applied in 12 load steps amounting to three times the coercive electric field, then lowered back to ground. In the second step, a vertical tip force of $\SI{16}{\newton}$ is applied to the tip surface. Due to the compression in the upper part of the beam, the material depolarizes mechanically in this section. 

The obtained values cannot be compared directly to the original work of \cite{ZouariZinebBenjeddou:2011}, as in this reference a different material model based on Kamlah's work \cite{Kamlah:2001} is used. However, we chose material constants close to their values, using $E_0 = \SI{1000}{\volt\per\milli\meter}$, $P_0 = \SI{0.3}{\coulomb\per\meter\squared}$, $S_0 = 0.002$, $m=1.1$, $\epsilon = \SI{1.5e-8}{\coulomb\per\volt\per\meter}$, $E_Y = 10^{6}\, \si{\newton\per\meter\squared}$, $\nu = 0.3$, $d_{31} = \SI{-2.74e-10}{\meter\per\volt}$, $d_{33} = \SI{5.93e-10}{\meter\per\volt}$, $H_0 = 10^{6}\, \si{\volt\meter\per\coulomb}$.  The $d_{15}$ effect was neglected in the current implementation. The regularization parameter from Section~\ref{sec:impl} is set to $\varepsilon = P_0 \times 10^{-4}$.

We used two different unstructured tetrahedral meshes -- a coarse one consisting of 141 elements, and a fine one consisting of 6848 elements. For the reference solution, we chose order $k=1$ as described in Section~\ref{sec:fespaces} on the fine mesh. This means second order displacement elements and first order polarization/dielectric displacements and leads to a total of 167\,451 degrees of freedom. The absolute value of the remanent polarization $|\Pv^{i}|$ after bending, and the corresponding stress distribution $T_{xx}$, are depicted in Figure~\ref{fig:cantilever}. As observed in \cite{ZouariZinebBenjeddou:2011}, the cantilever depolarizes in the region close to the clamped end where compressive stresses arise. 
This depolarization reduces the stress level there to a maximum of $-128.22~\si{\newton\per\milli\meter\squared}$, which compares well to the values listed in the original reference. However, due to the different description of the ferroelectric material, we observe stronger depolarization.

We compare our results to a second computation on the very coarse mesh using high order $k=2$, i.e. third order displacement elements and second order polarizations/dielectric displacements. In this case, we end up with 9653 degrees of freedom in total, while maintaining good accuracy (see Figure~\ref{fig:cantilever_coarse}). Note that, in both cases, neither polarization nor stresses have been post-processed in any way, but the finite element solution is depicted directly.
\oldnew{}{In both cases, the vertical tip force was added in eight equal-sized load steps. In each load steps, between 8 and 11 Newton iterations had to be done in order to reduce the $\ell^2$-norm of the residual by a factor of $10^{-6}$. These iteration counts compare well to counts provided by \cite{SemenovLiskowskyBalke:2010} for their return-mapping algorithm using correct tangent moduli.}

In Figure~\ref{fig:pix}, we plot the distribution of the $x$-component of the (scaled) irreversible polarization $P^{i}_x/P_0$ over the central line on top of the cantilever $\{ (x, y=b/2, z=h)\}$. In comparison to the values presented by Zouari et al. \cite{ZouariZinebBenjeddou:2011}, we observe that the polarization drops further to approximately 65\% as compared to saturation. Similar to their findings, we see an almost linear distribution of the remanent polarization along the major part of the cantilever, and a strong decrease close to the clamped end.

\begin{figure}
	\begin{center}
		\includegraphics[width=0.49\textwidth]{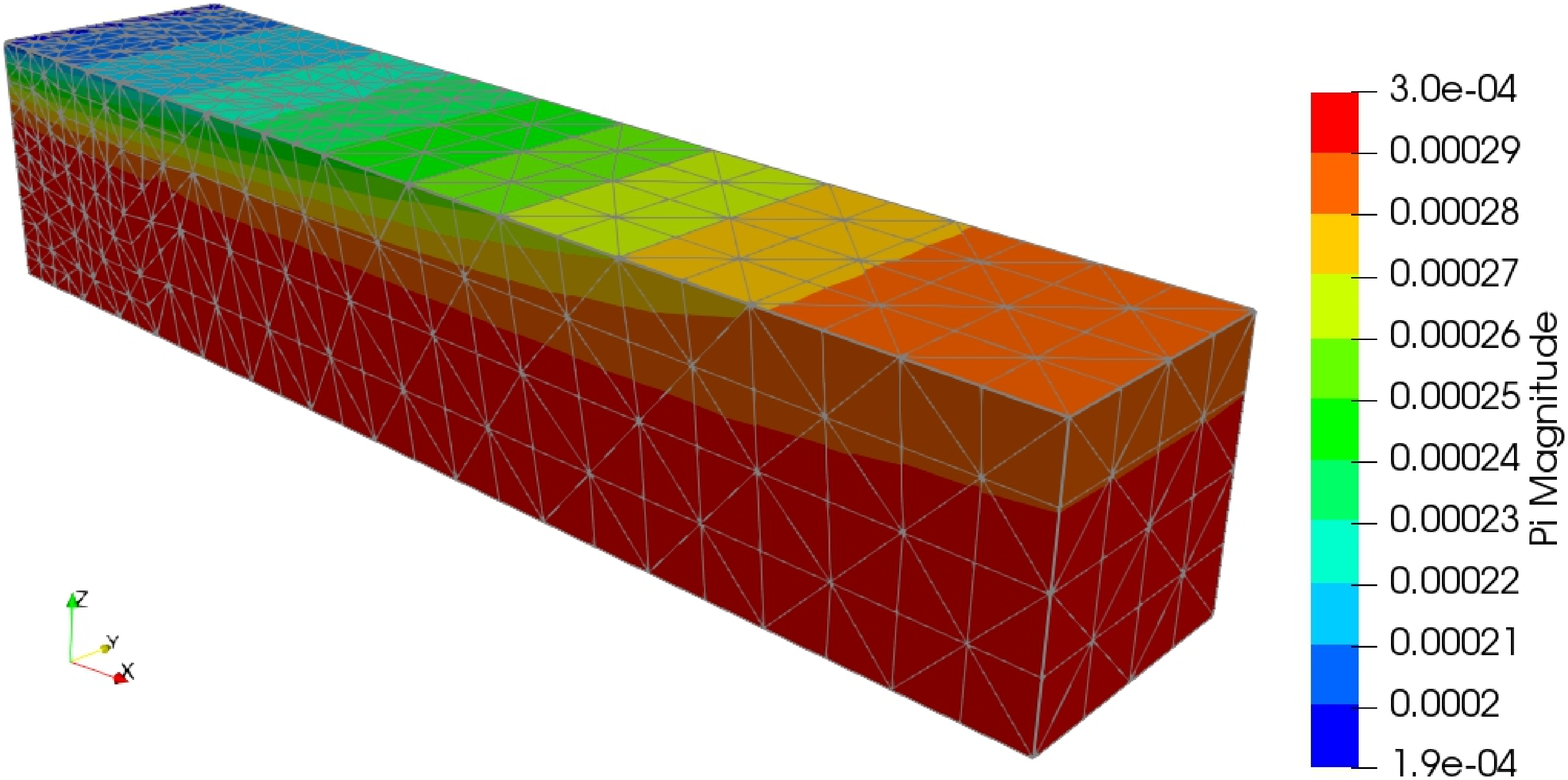}
		\includegraphics[width=0.49\textwidth]{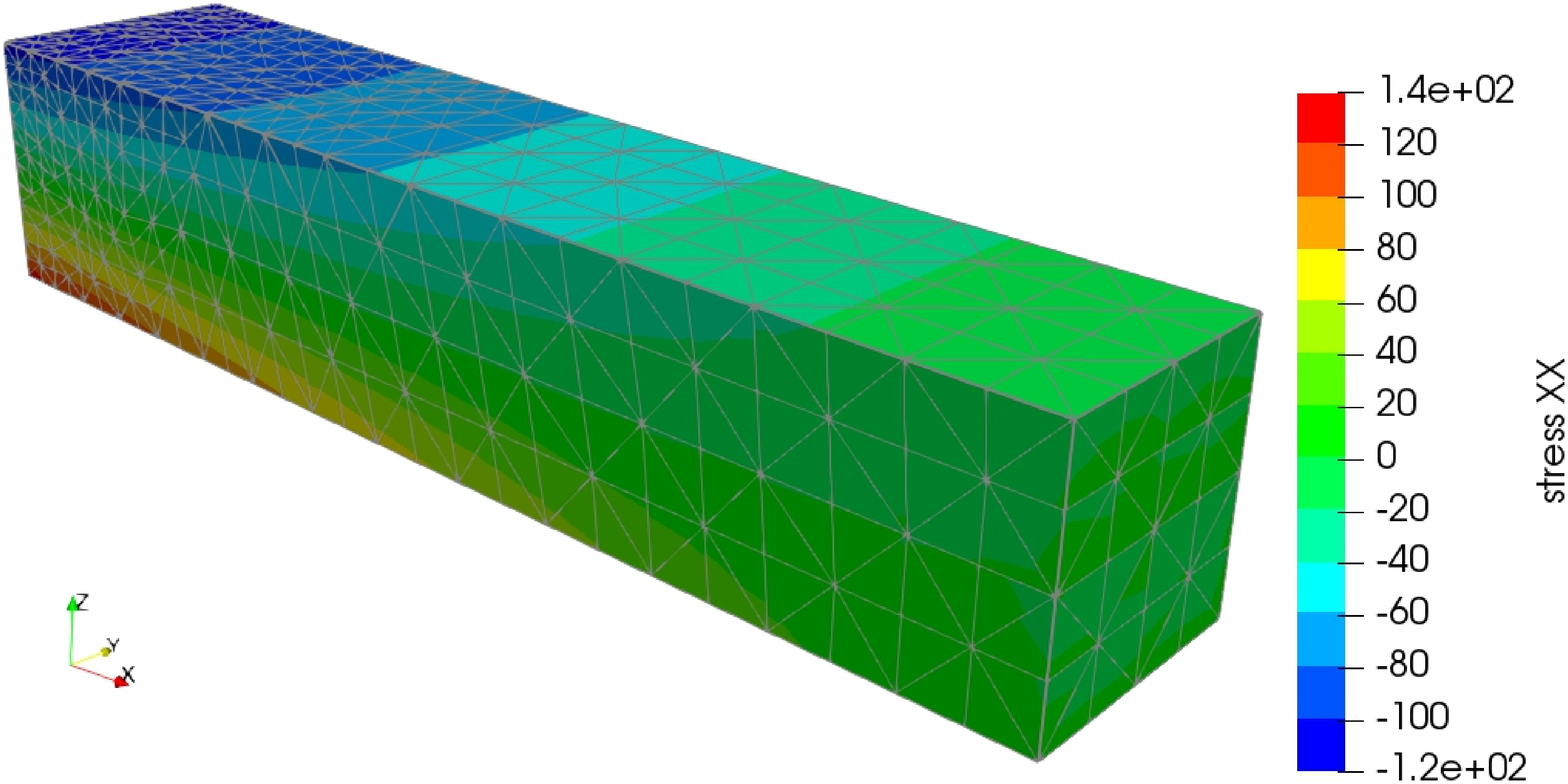}
	\end{center}
	\caption{Mechanical depolarization of a polarized cantilever beam under a vertical tip force -- absolute value of remanent polarization $|\Pv^{i}|$ (left) and stress distribution $T_{xx}$. Finite elements as described in Section~\ref{sec:fespaces} for order $k=1$ are used.} \label{fig:cantilever}
\end{figure}
\begin{figure}
	\begin{center}
		\includegraphics[width=0.49\textwidth]{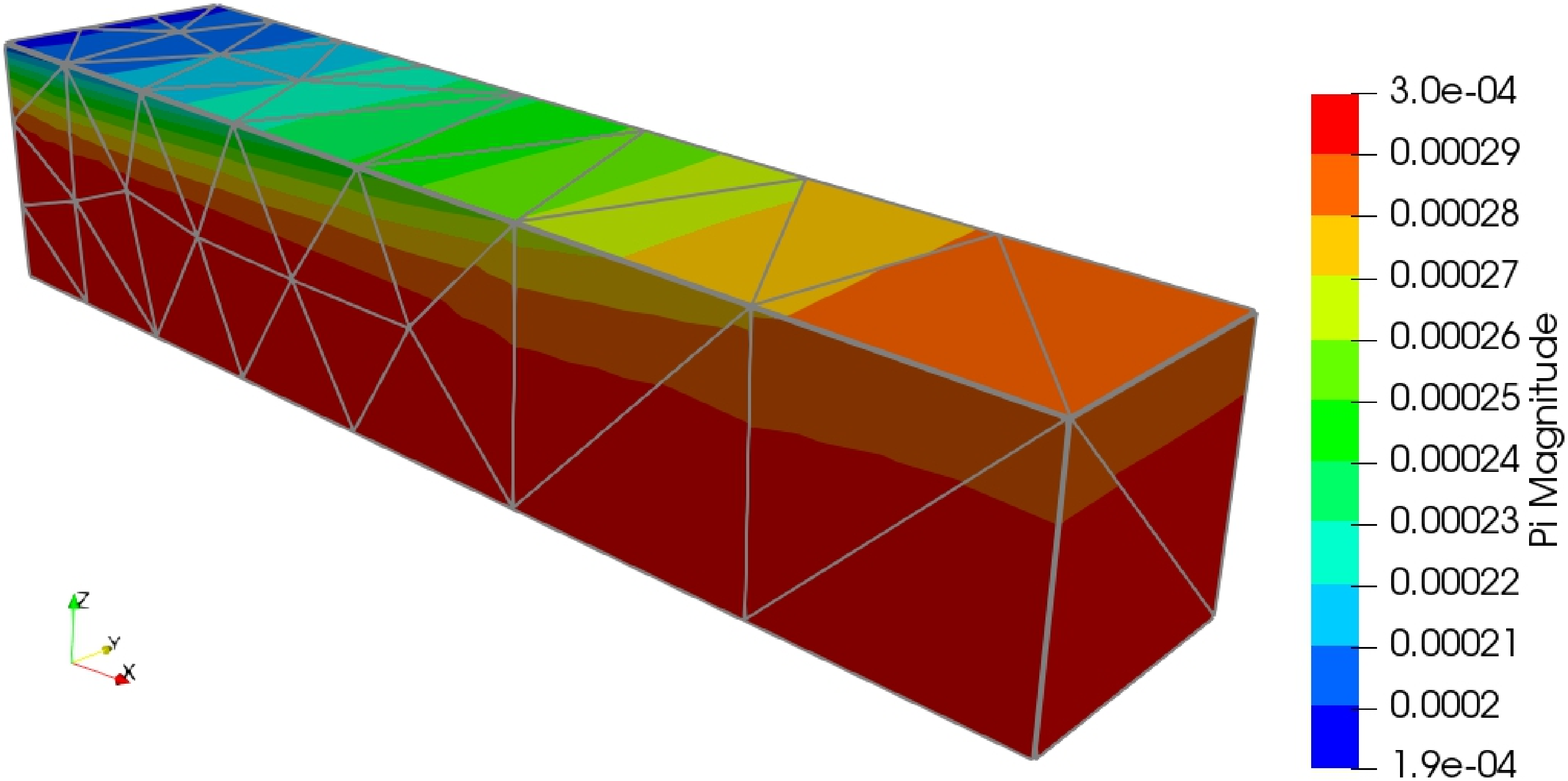}
		\includegraphics[width=0.49\textwidth]{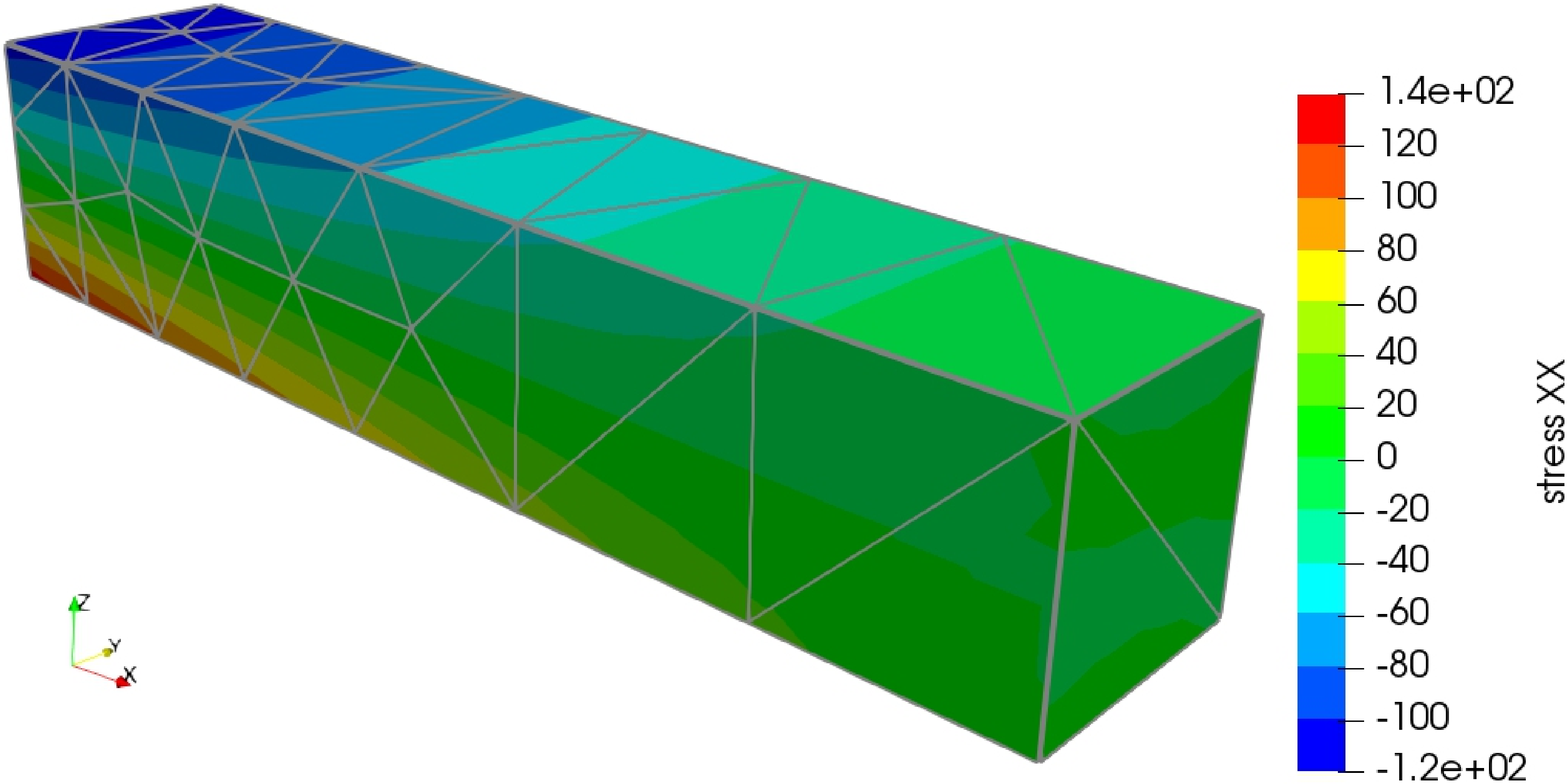}
	\end{center}
	\caption{Mechanical depolarization of a polarized cantilever beam under a vertical tip force -- absolute value of remanent polarization $|\Pv^{i}|$ (left) and stress distribution $T_{xx}$. Finite elements as described in Section~\ref{sec:fespaces} for order $k=2$ are used.} \label{fig:cantilever_coarse}
\end{figure}
\begin{figure}
	\begin{center}
		\includegraphics[width=0.49\textwidth]{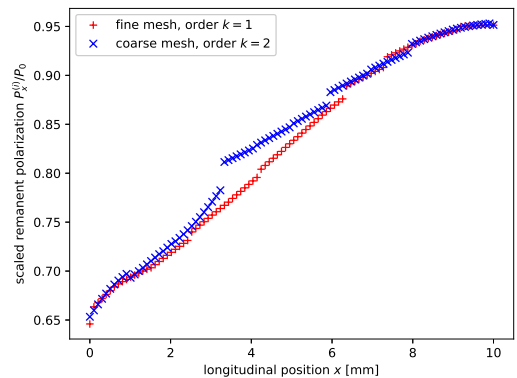}
	\end{center}
	\caption{Mechanical depolarization of a polarized cantilever beam under a vertical tip force -- remanent polarization $P^{i}_x/P_0$ on the top surface along the axial coordinate for different discretizations and finite element orders $k=1$ and $k=2$.} \label{fig:pix}
\end{figure}

\section{Conclusion}
In this contribution, we have formulated the polarization problem in ferroelectric media as a variational inequality. In this framework, we were able to show existence and uniqueness of a solution to the time-discrete update problem under reasonable assumptions on the free energy. Under stronger assumptions, it is possible to prove existence of a solution to the time-dependent problem. We propose to choose finite elements such that these assumptions are satisfied also in the discretized setting. To solve the discrete problems, we regularize the non-differentiable dissipation function. Then is is possible to solve the optimization problem all at once by a single Newton iteration. All numerical results provided in this contributions have been generated in the open-source software package Netgen/NGSolve, which provides all the non-standard elements of arbitrary order as well as automatic differentiation of the energies.

\section{Acknowledgements}
Martin Meindlhumer acknowledges support of Johannes Kepler University Linz, Linz Institute of Technology (LIT).\\
This work has been supported by the Linz Center of Mechatronics (LCM) in the framework of the Austrian COMET-K2 program.

\bibliographystyle{plain}
\bibliography{Polarization}

\end{document}